\documentclass[ejs,preprint]{imsart}

\RequirePackage[OT1]{fontenc}
\RequirePackage{amsmath}
\RequirePackage[numbers]{natbib}
\RequirePackage[colorlinks,citecolor=blue,urlcolor=blue]{hyperref}

\usepackage{amsmath}
\usepackage{mathtools}
\usepackage{amssymb}
\RequirePackage[numbers]{natbib}
\RequirePackage[colorlinks,citecolor=blue,urlcolor=blue]{hyperref}

\doi{10.1214/154957804100000000}
\pubyear{0000}
\volume{0}
\firstpage{0}
\lastpage{0}
\startlocaldefs
\endlocaldefs

\RequirePackage[OT1]{fontenc}
\RequirePackage[numbers]{natbib}
\RequirePackage[colorlinks,citecolor=blue,urlcolor=blue]{hyperref}
\usepackage{amsfonts}

\usepackage[utf8]{inputenc}
\usepackage[english]{babel}

\usepackage{graphicx}
\usepackage{amsmath}
\usepackage{subcaption}
\mathtoolsset{showonlyrefs,showmanualtags}
\usepackage{bbm}
\usepackage{amsthm}
\usepackage{mathrsfs}
\newcommand{\R}{\mathbb{R}}
\newcommand{\N}[2]{\mathcal{N}(#1, #2)}
\newcommand{\Nat}{\mathbb{N}}
\newcommand{\E}[1]{\mathbb{E}\left[#1\right]}
\newcommand{\Var}[1]{\mathrm{Var}\left[#1\right]}
\newcommand{\Varboot}[1]{\mathrm{Var}_\flat\left[#1\right]}
\newcommand{\Ynb}{Y_n^\flat}
\newcommand{\trueSigma}{\Sigma^*}
\newcommand{\hatSigma}{\hat\Sigma}
\newcommand{\trueTheta}{\Theta^*}
\newcommand{\hatTheta}{\hat\Theta}
\newcommand{\xn}{x_n}
\newcommand{\zb}{z^\flat}
\newcommand{\zm}{z_-}
\newcommand{\znb}{z^{N^\flat}}
\newcommand{\zp}{z_+}

\newcommand{\deltaZ}{\Delta_Z(x)}
\newcommand{\probXBound}{p_s^X(x)}
\newcommand{\bZp}{\partial\mathcal{Z}_+}
\newcommand{\bZm}{\partial\mathcal{Z}_-}
\newcommand{\Zcb}{\mathcal{Z}^\flat}
\newcommand{\Is}{\mathcal{I}_n^{\mathfrak{S}}(t)}
\newcommand{\Ir}{\mathcal{I}^r_n(t)}
\newcommand{\Il}{\mathcal{I}^l_n(t)}
\newcommand{\Istable}{\mathcal{I}_s}
\newcommand{\infnorm}[1]{\left|\left|#1\right|\right|_\infty}
\newcommand{\InfNorm}[1]{\left|\left|\left|#1\right|\right|\right|_\infty}
\newcommand{\normtwo}[1]{\left|\left|#1\right|\right|_2}
\newcommand{\inv}[1]{#1^{-1}}
\newcommand{\Ziuv}{Z_{i,uv}}
\newcommand{\hatZiuv}{\hat Z_{i,uv}}
\newcommand{\Zbiuv}{Z_{i,uv}^\flat}
\newcommand{\Zbi}{Z_{i}^\flat}
\newcommand{\Prob}[1]{\mathbb{P} \left\{#1\right\}}
\newcommand{\Pb}{P^\flat}
\newcommand{\Nb}{\mathcal{N}^\flat}
\newcommand{\NN}{\mathcal{N}}
\newcommand{\Zp}{\mathcal{Z}_+}
\newcommand{\Zm}{\mathcal{Z}_-}

\newcommand{\Probboot}[1]{\mathbb{P^\flat} \left\{#1\right\}}
\newcommand{\Anb}{A_n^\flat}
\newcommand{\TauT}{\mathcal{T}_T}
\newcommand{\alphacorrected}{\alpha^*}
\newcommand{\An}{A_n}
\newcommand{\Sa}{\mathcal{S}}
\newcommand{\xalphan}{x^\flat_n(\alpha)}
\newcommand{\alphap}{\alpha^+}
\newcommand{\alpham}{\alpha^-}

\newcommand{\q}{\mathrm{q}} 
\newcommand{\Tnmax}{\mathbb{T}_{\nmax}} 
\newcommand{\Tauarr}{\mathcal{T}_{\leftrightarrow}}
\newcommand{\xalphann}[1]{x^\flat_{#1}(\alpha)}
\newcommand{\abs}[1]{\left|#1\right|}
\newcommand{\ex}[1]{e^{#1}}
\newcommand{\onenorm}[1]{\left|\left|#1\right|\right|_1}

\newcommand{\Zbound}[2]{\mathcal{Z}_{#1}(#2)}

\newcommand{\Zboundsqr}[2]{\mathcal{Z}^2_{#1}(#2)}

\newcommand{\probzbound}[2]{p_{\mathcal{Z}_{#1}}(#2)}

\newcommand{\probsigmazboundone}{p^{\Sigma_{Z 1}}_{_{s}}(x, q)}
\newcommand{\probsigmazboundtwo}{p^{\Sigma_{Z 2}}_{_{s}}(x)}
\newcommand{\probsigmay}{p^{\Sigma_Y}_{s}(x,q)}
\newcommand{\probSigmaBound}[2]{p^{\Sigma}(#2)}

\newcommand{\Zv}{\overline{Z}}

\newcommand{\Zhatvi}{\overline{\hat Z_i}}
\newcommand{\trueSigmaZ}{\Sigma^*_Z}
\newcommand{\hatSigmaZ}{\hat\Sigma_Z}

\newcommand{\hatSigmahatZ}{\hat\Sigma_{\hat Z}}
\newcommand{\Wi}{W^{(i)}}
\newcommand{\Sz}{S_Z}
\newcommand{\Sbz}{S_Z^\flat}
\newcommand{\Snbz}{S_Z^{n\flat}}
\newcommand{\Szn}{S_Z^n}

\newcommand{\Znormalized}{Z^{{S}}}

\newcommand{\Zbnormalized}{Z^{{S\flat}}}
\newcommand{\trueSigmaY}{\Sigma^*_Y}
\newcommand{\hatSigmaY}{\hat\Sigma_Y}
\newcommand{\mymin}[2]{\min_{#1}\left[#2\right]}
\newcommand{\kappaGamma}{\kappa_{\Gamma^*}}
\newcommand{\kappaSigma}{\kappa_{\Sigma^*}}

\newcommand{\Ra}{R_A}
\newcommand{\Xb}{X^\flat}
\newcommand{\bootTheta}{\hat\Theta}
\newcommand{\Ab}{A^\flat}

\newcommand{\Yb}{Y^\flat}
\newcommand{\empE}[1]{\mathbb{E}_{s}\left[#1\right]}

\newcommand{\Bernstain}[4]{\frac{\left(#1\right)#2}{3#3}\left(1+ \sqrt{1+\frac{9#3#4}{#2\left(#1\right)^2}} \right)}

\newcommand{\SigmaZDelta}{\Delta_{\Sigma_Z}}

\newcommand{\Rboot}{R_{A^\flat}}

\newcommand{\Rsigma}{R_{\Sigma}^{\pm}}

\newcommand{\probMoment}[2]{p^M_s(#1)}
\newcommand{\CA}{C_{A}}
\newcommand{\CAb}{C_{A^\flat}}
\newcommand{\hatCAb}{\hat{C}_{A^\flat}}
\newcommand{\deltaY}{\Delta_Y}
\newcommand{\IVar}[1]{\mathrm{Var}\left[#1\right]^{-1}}
\newcommand{\Zvi}{\overline{Z_i}}
\newcommand{\hatZvi}{\overline{\hat Z_i}}
\newcommand{\n}{\mathfrak{N}}
\newcommand{\Tn}{\mathbb{T}_n}
\newcommand{\nmin}{n_-}
\newcommand{\nmax}{n_+}
\newcommand{\numn}{\abs{\mathfrak{N}}}
\newcommand{\glestimation}{\hatTheta^{GL}}
\newcommand{\mbestimation}{\hatTheta^{MB}}
\newcommand{\tr}{tr}
\newcommand{\maxLambda}[1]{\Lambda\left(#1\right)}
\newcommand{\minLambda}[1]{\lambda\left(#1\right)}
\newcommand{\rank}{rank}
\newcommand{\Tau}{\mathcal{T}}
\newcommand{\xii}{\xi_i}
\newcommand{\hatxii}{\hat\xi_i}
\newcommand{\x}{\mathrm{x}}
\newcommand{\twoNorm}[1]{\lvert\lvert#1\rvert\rvert_{2}}

\newcommand{\Hnull}{\mathbb{H}_0}
\newcommand{\Halt}{\mathbb{H}_1}
\newcommand{\Rt}{R_{\hat T}}
\newcommand{\s}{\mathfrak{S}}

\startlocaldefs
\numberwithin{equation}{section}
\theoremstyle{plain}
\newtheorem{theorem}{}[section]
\newtheorem{corollary}{}[section]
\newtheorem{lemma}{}[section]
\newtheorem{definition}{}[section]
\newtheorem{assumption}{}[section]
\newtheorem{remark}{}[section]
\endlocaldefs

\begin{document}

\renewcommand{\thedefinition}{Definition \thesection.\arabic{definition}}
\renewcommand{\thelemma}{Lemma \thesection.\arabic{lemma}}
\renewcommand{\thetheorem}{Theorem \thesection.\arabic{theorem}}
\renewcommand{\theassumption}{Assumption \thesection.\arabic{assumption}}
\renewcommand{\theremark}{Remark \thesection.\arabic{remark}}
\renewcommand{\thecorollary}{Corollary \thesection.\arabic{corollary}}

\begin{frontmatter}

% "Title of the Paper"
\title{Change-point detection in high-dimensional covariance structure}
\runtitle{Change-point detection in high-dimensional covariance structure}

% indicate corresponding author with \corref{}
% \author{\fnms{John} \snm{Smith}\thanksref{t2}\corref{}\ead[label=e1]{smith@foo.com}\ead[label=e2,url]{www.foo.com}}
% \thankstext{t2}{Thanks to somebody} 
% \address{line 1\\ line 2\\ \printead{e1}\\ \printead{e2}}
\begin{aug}
	\author{\fnms{Valeriy} \snm{Avanesov}\corref{}\ead[label=e1]{avanesov@wias-berlin.de}},
	\author{\fnms{Nazar} \snm{Buzun}\ead[label=e2]{buzun@wias-berlin.de}}
	
	\address{WIAS\\Mohrenstr. 39\\ 
		10117 Berlin\\ 
		Germany\\
		\printead{e1,e2}}

	\runauthor{V. Avanesov, N. Buzun}
	\affiliation{WIAS}
	
\end{aug}

\begin{abstract}
	  ~In this paper we introduce a novel approach for an important problem of break detection. Specifically, we are interested in detection of an abrupt change in the covariance structure of a high-dimensional random process  -- a problem, which has applications in many areas e.g., neuroimaging and finance. The developed approach is essentially a testing procedure involving a choice of a critical level. To that end a non-standard bootstrap scheme is proposed and theoretically justified under mild assumptions. Theoretical study features a result providing guaranties for break detection. All the theoretical results are established in a high-dimensional setting (dimensionality $p \gg n$). Multiscale nature of the approach allows for a trade-off between sensitivity of break detection and localization. The approach can be naturally employed in an on-line setting. Simulation study demonstrates that the approach matches the nominal level of false alarm probability and exhibits high power, outperforming a recent approach.
\end{abstract}



% history:
% \received{\smonth{1} \syear{0000}}

\end{frontmatter}

\section{Introduction} 
The analysis of high dimensional time series is crucial for many fields including neuroimaging and financial engineering.
There, one often has  to deal with processes involving abrupt structural changes which necessitate a corresponding adaptation of a model and/or a strategy. Structural
break analysis comprises determining if an abrupt change is present in the given sample and if so,
estimating the change-point, namely the moment in time when it takes place. In literature both
problems may be referred to as {\it change-point} or {\it break detection}. In this study we will be using
terms {\it break detection }and {\it change-point localization }respectively in order to distinguish between
them.
The majority of approaches to the problem consider only a univariate process \cite{limitTheoremsCPAnalysis} \cite{aue2013}.
However, in recent years the interest for multi-dimensional approaches has increased. Most of them cover the case of fixed dimensionality \cite{Matteson2015} \cite{Lavielle2006} \cite{aue2009} \cite{xie2013} \cite{zou2014}. Some approaches \cite{cho2016, jirak2015, Cho2015} feature {\it high-dimensional} theoretical guaranties but only the case of dimensionality polynomially growing in sample size is covered. The case of exponential growth has not been considered so far.

In order to detect a break, a test statistic is usually computed for each point $t$ (e.g. \cite{Matteson2015}). The break is detected if the maximum of these values exceeds a certain \textit{threshold}. A proper choice of the latter may be a tricky issue. Consider a pair of plots (Figure \ref{figgg}) of the statistic $A(t)$ defined in Section \ref{apprsec}. It is rather difficult to see how many breaks are there, if any. The classic approach to the problem is based on the asymptotic behaviour of the statistic \cite{limitTheoremsCPAnalysis} \cite{aue2013} \cite{aue2009} \cite{jirak2015} \cite{biau2016} \cite{zou2014}. As an alternative, permutation \cite{jirak2015} \cite{Matteson2015} or parametric bootstrap may be used \cite{jirak2015}. Clearly, it seems attractive to choose the threshold in a solely data-driven way as it is suggested in the recent paper \cite{cho2016}, but a careful bootstrap validation is still an open question.

In the current study we are interested in a particular kind of a break -- an abrupt transformation in the inverse covariance matrix -- which is motivated by applications to neuroimaging. The covariance structure of data in functional Magnetic Resonance Imaging has recently drawn a lot of interest, as it encodes so-called functional connectivity networks \cite{sporns2011} which refer to the explicit influence among neural systems \cite{journals/brain/Friston11}. A rather popular approach to  inferencing these networks is based on estimating inverse covariance or precision matrices \cite{Allen2012}. The technique generally makes use of the observation that functional connectivity networks are of small-world type  \cite{sporns2011}, which makes sparsity assumptions feasible.  Analysing the dynamics of these networks is  particularly important for the research on neural diseases and also in the context of brain development with emphasis on characterizing the re-configuration of the brain during learning  \cite{Bassett_Wymbs_Porter_Mucha_Carlson_Grafton_2010}.

A similar problem is found in finance: the dynamics of the covariance structure of a high-dimensional process modelling exchange rates and market indexes is crucial for a proper asset allocation in a portfolio \cite{Serban2007,Bauwens2006,Engle1990,Mikosch2009}.

One approach to the  change-point localization is developed in \cite{Lavielle2006}, the corresponding significance testing problem is considered in \cite{aue2009}. However, neither of these papers address the high-dimensional case.

A widely used break detection approach (named CUSUM) \citep{Cho2015, aue2009, jirak2015} suggests to compute a statistic at a point $t$ as a distance of estimators of some parameter of the underlying distributions obtained using all the data before and after that point. This technique requires the whole sample to be known in advance, which prevents it from being used in \textit{online} setting. In order to overcome this drawback we propose a method ideologically similar to MOSUM \cite{Bauer1980} \cite{eichinger2018}: choose a window size $n \in \Nat$ and compute parameter estimators using only $n$ points before and $n$ points after the \textit{central point} $t$ (see Section \ref{apprsec} for formal definition). Window size $n$ is an important parameter and its choice is case-specific (see Section \ref{precsensitivitysec} for theoretical treatment of this issue). Using a small window results in high variability and low sensitivity, while a large window implies higher uncertainty in change-point localization yielding the issue of a proper choice of window size. The \textit{multiscale} nature of the proposed method enables us to 	incorporate the advantages of narrower and wider windows by considering multiple window sizes at once in order for wider windows to provide higher sensitivity while narrower ones improve change-point localization. Moreover, the local nature of the proposed statistic allows for detection of multiple change points if the change-points are not too close to each other (see Section \ref{precsensitivitysec}).

The contribution of our study is the development of a novel  break detection approach which is 

\begin{itemize}
	\item high-dimensional, allowing for up to exponential growth of the dimensionality with the window size
	\item suitable for on-line setting
	\item suitable for detection of multiple change-points
	\item multiscale, attaining trade-off between break detection sensitivity and change-point localization accuracy
	\item using a fully data-driven threshold selection algorithm rigorously justified under mild assumptions
	\item featuring formal sensitivity guaranties in high-dimensional setting
\end{itemize} 

\begin{figure}[!tbp]
	\centering
	\begin{minipage}[b]{0.48\textwidth}
		\includegraphics[width=\textwidth]{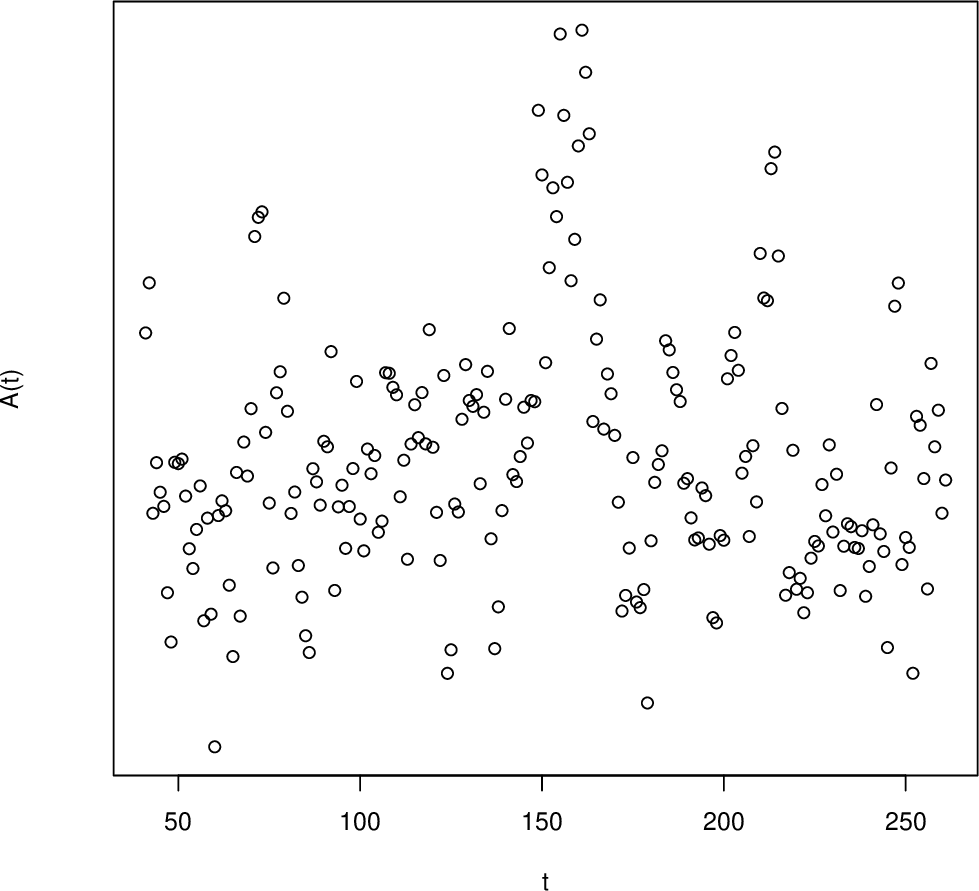}
	\end{minipage}
	\hfill
	\begin{minipage}[b]{0.48\textwidth}
		\includegraphics[width=\textwidth]{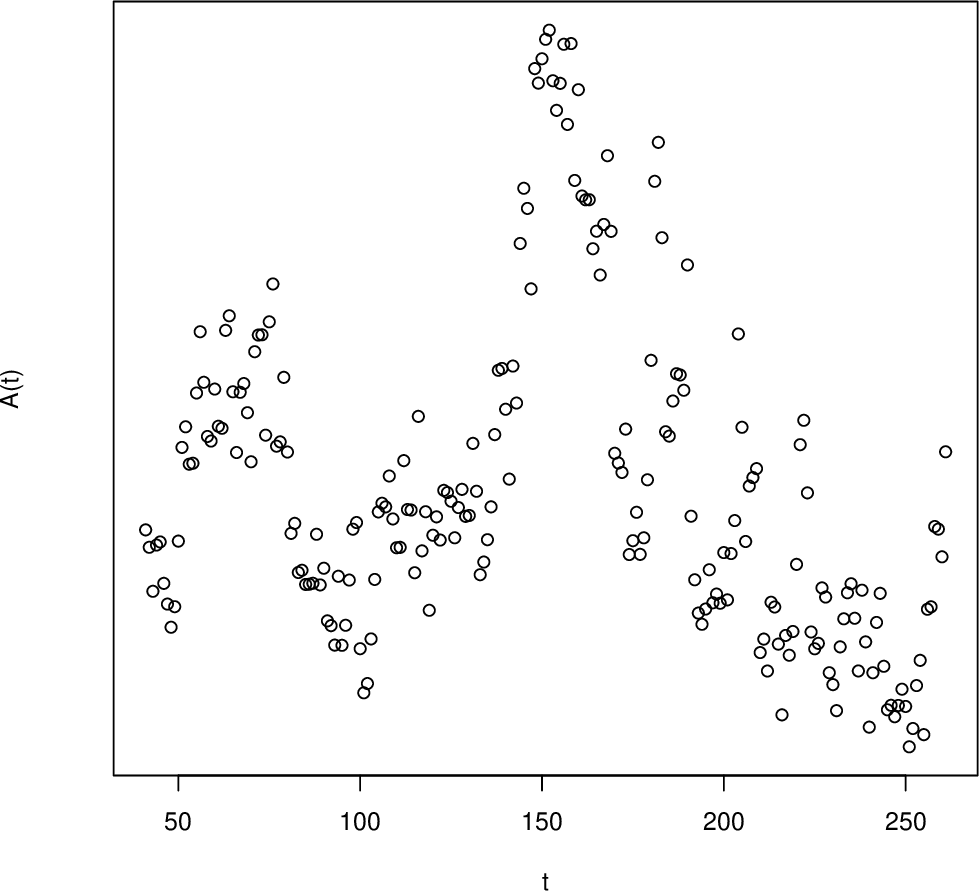}
	\end{minipage}
	\caption{Plots of test statstics $A(t)$ computed on synthetically generated data without (left) and with a single change-point at $t=150$ (right). The ticks and their corresponding values are intentionally hidden for the vertical axis, since we aim to draw a conclusion about a brake based only on a single time series, not a collection of them. Clearly, the choice of a threshold is not obvious. }
	\label{figgg}
\end{figure}

We consider the following setup. Let $X_1, ...~X_N~\in~\R^p$ denote sample of independent zero-mean vectors  (the on-line setting is discussed in Section \ref{online}) and we want to test a hypothesis 
\begin{equation}
\begin{split}
\Hnull &\coloneqq \{\forall i : \IVar{X_i} = \IVar{X_{i+1}}   \} 
\end{split}
\end{equation}
versus an alternative suggesting the existence of a break:
\begin{equation}\begin{split}
\Halt  \coloneqq   \left\{ \exists \tau :  \IVar{X_1} \right. & \left. = \IVar{X_2} = ... = \IVar{X_{\tau}} \right. \\ & \neq  \left. \IVar{X_{\tau+1}} = ... = \IVar{X_N} \right\}\\
\end{split}
\end{equation}
and localize the change-point $\tau$ as precisely as possible or (in on-line setting) to detect a break as soon as possible.  

The approach proposed in the paper focuses on applications in neuroimaging. Independence of the vectors $X_i$ is fulfilled only approximately in practice but functional connectivity network analysis, which assumes temporal independence, has been proven to be very successful and validated \cite{Poldrack2011}.

In the current study it is also assumed that some subset of indices $\Istable \subseteq 1..N$ of size $s$ (possibly, $s=N$) is chosen. The threshold is chosen relying on the sub-sample $\{X_i\}_{i \in \Istable}$ while the test-statistic is computed based on the whole sample.

To this end we define a family of test statistics in Section \ref{apprsecStat} which is followed by Section \ref{apprsecBoot} describing a data-driven (bootstrap) calibration scheme and Section \ref{locsec} describing change-point localization procedure. The theoretical part of the paper justifies the proposed procedure in a high-dimensional setting. The result justifying the validity of the proposed calibration scheme is stated in Section \ref{mainressec}. Section \ref{precsensitivitysec} is devoted to the sensitivity result yielding a bound for the window size $n$ necessary to reliably detect a break of a given extent and hence bounding the uncertainty of the change-point localization (or the delay of detection in online setting). The theoretical study is supported by a comparative simulation study (described in Section \ref{simsec}) demonstrating conservativeness of the proposed test and higher sensitivity compared to the other algorithms and by analysis of real-world datasets (Section \ref{real}). Appendix \ref{sensitivityPrecisionMatrixResult} contains a finite-sample version of sensitivity result along with the proofs. Appendix \ref{mainProof} provides a a finite-sample version of bootstrap sensitivity result which is followed by the proofs. Finally, Appendix \ref{knownresults} lists results which were essential for our theoretical study.

\section{Proposed approach} \label{apprsec} This section describes the proposed approach along with a data-driven calibration scheme.
Informally the proposed statistic can be described as follows. Provided that the break may happen only at moment $t$, one could estimate some parameter of the distribution using $n$ data-points to the left of $t$, estimate it again using $n$ data-points to the right and use the norm of their difference as a test-statistic $A_n(t)$. Yet, in practice one does not usually possess such knowledge, therefore we propose to maximize these statistics over all possible locations $t$ yielding $A_n$. Finally, in order to attain a trade-off between break detection sensitivity and change-point localization accuracy we build a multiscale approach: consider a family of test statistics $\{A_n\}_{n \in \n}$ for multiple window sizes $n \in \n \subset \Nat$ at once.

\subsection{Definition of the test statistic} \label{apprsecStat}

Now we present a formal definition of the test statistic.
In order to detect a break we consider a set of window sizes $\n \subset \mathbb{N}$. Denote the size of the widest window as $\nmax$ and of the narrowest as $\nmin$. Given a sample of length $N$, for each window size $n \in \n$ define a set of central points $\Tn \coloneqq \{n+1,..., N-n+1\}$.  Next, for all $n \in \n$ define a set of indices which belong to the window on the left side of the central point $t \in \Tn$ as $\Il \coloneqq \{t-n, ... , t-1\}$ and correspondingly for the window on the right side define $\Ir \coloneqq \{t, ... , t+n-1\}$.
Denote the sum of numbers of central points for all window sizes $n \in \n$ as

\begin{equation}
T \coloneqq \sum_{n \in \n} \abs{\Tn}.
\end{equation}
For each window size $n \in \n$, each central point $t \in \Tn$ and each side $\s \in \{l,r\}$ we define a de-sparsified estimator of precision matrix \cite{sara} \cite{janamb} as 

\begin{equation}\label{Tdefinition}
\begin{split}
\hat T^{\s}_n(t)  &\coloneqq \hatTheta^{\s}_n(t) + \hatTheta^{\s}_n(t)^T- \hatTheta^{\s}_n(t)^T\hatSigma^{\s}_n(t)\hatTheta^{\s}_n(t) 
\end{split}\end{equation}
where

\begin{equation}
\hatSigma^\s_n(t) = \frac{1}{n} \sum_{i \in \Is} X_i X_i^T 
\end{equation}
and $\hatTheta^\s_n(t)$ is a consistent estimator of precision matrix which can be obtained by graphical lasso \cite{ravikumar2011} or node-wise procedure \cite{janamb} (see \ref{consdef} and Appendix \ref{glassosec} for details). Note, the symmetricity of $\hatTheta^\s_n(t)$ is not required, yet $\hat T^{\s}_n(t)$ is symmetric by construction.

Now define a matrix of size $p \times p$ with elements

\begin{equation}
\label{zijk}
\Ziuv := \trueTheta_u X_i\trueTheta_vX_i-\trueTheta_{uv}
\end{equation}
where $\trueTheta \coloneqq \inv{\E{X_i X_i^T}}$ for $i \le \tau$, $\trueTheta_u$ stands for the $u$-th row of $\trueTheta$. Denote their variances as  $\sigma^2_{uv} \coloneqq \Var{Z_{1, uv}}$ and introduce the diagonal matrix $S = diag(\sigma_{1,1},\sigma_{1,2}, ... , \sigma_{p,p-1}, \sigma_{p,p} )$.
Denote a consistent estimator (see \ref{consdef} for details) of the precision matrix $\trueTheta$ obtained based on the sub-sample $\{X_i\}_{i \in \Istable}$ as  $\bootTheta$ .
In practice, the variances $\sigma^2_{uv}$ are unknown, but under normality assumption one can plug in $\hat\sigma_{uv}^2 \coloneqq \hatTheta_{uu} \hatTheta_{vv} + \hatTheta_{uv}^2$ which have been proven to be  consistent (uniformly for all $u$ and $v$) estimators of $\sigma^2_{uv}$ \cite{sara} \cite{mypaper}. If the node-wise procedure is employed, the uniform consistency of an empirical estimate of $\sigma^2_{uv}$ has been shown under some mild assumptions (not including normality) \cite{janamb}.

For each window size $n \in \n$  and a central point $t \in \Tn$ we define a statistic

\begin{equation}  \label{Adef}
\begin{split}
A_n(t) &\coloneqq \infnorm{\sqrt{\frac{n}{2}}\inv{S}\overline{(\hat T^l_n(t) - \hat T^r_n (t))}} 
\end{split}
\end{equation}
where we write $\overline{M}$ for a vector composed of stacked columns of matrix $M$.
Finally we define our family of test statistics for all $n \in \n$ as 

\begin{equation}
A_n = \max_{\substack{t \in \Tn}} A_n(t).
\end{equation}
Our approach heavily relies on the following expansion under $\Hnull$

\begin{equation} \label{Texpantion}
\sqrt{n} (\hat T^{\s}_n(t) - \trueTheta) = \frac{1}{\sqrt{n}}\sum_{i \in \Is} Z_i + {r_n^{\s}(t)\sqrt{n}},
\end{equation}
where the residual term 
\begin{equation}
r^{\s}_n(t) \coloneqq \hat T^{\s}_n(t) - \left( \trueTheta - \trueTheta \left(\hatSigma^\s_n(t) - \trueSigma\right) \trueTheta\right)
\end{equation}
can be controlled under mild assumptions \cite{sara} \cite{janamb} \cite{mypaper} as

\begin{equation}\label{key}
\max_{\substack{\s \in \{l,r\} \\ n \in \n, t \in \Tn}} \infnorm{r^{\s}_n(t)} = O_P\left(\frac{d \log p}{n}\right).
\end{equation}

The main reason why we prefer to use de-sparsified estimators $\hat T_n^\s(t)$  over using $\ell_1$-penalized estimators $\hatTheta_n^\s(t)$ is that the former allows for the expansion \eqref{Texpantion} which pitches an idea behind the bootstrap procedure we suggest and makes the theoretical analysis possible. 

This expansion might have been used in order to investigate the asymptotic properties of $A_n$ and obtain the threshold, however we propose a data-driven scheme.

\begin{remark} \label{remarkOnCovStat}
	A different test statistic can be defined as the maximum distance between elements of empirical covariance matrices $\hatSigma^l_n(t)$ and $\hatSigma^r_n(t)$. Such a method would be computationally less burdensome, since it does not involve precision matrix estimation. However, there are extremely effective implementations of the latter which makes the computational costs of calibration procedure dominate. In turn, the calibration complexity is the same for both of the approaches, since matrix inversion is not involved therein (see Section \ref{apprsecBoot}). Therefore, the computational gain would be negligible. Furthermore, application to neuroimaging motivates the search for a structural change in a functional connectivity network which is encoded by the structure of the corresponding precision matrix. Clearly, a change in the precision matrix also means a change in the covariance matrix, though we believe that the definition \eqref{Adef} increases the sensitivity to this kind of alternative (see \ref{breakextentremark} for more details).  
\end{remark}

\begin{remark}
	The estimator $\hat T_n^\s(t)$ is indeed a de-biased estimator. We can easily rearrange the definition \eqref{Tdefinition} as 
	
	\begin{equation}
	\hat T^{\s}_n(t) = \hatTheta^{\s}_n(t) - \hatTheta^{\s}_n(t)^T \left(  \hatSigma^{\s}_n(t)\hatTheta^{\s}_n(t)  - I\right)
	\end{equation}
	in order to represent it as a difference of $\ell_1$-penalized estimator and bias-correcting term.
\end{remark}

\subsection{Bootstrap calibration} \label{apprsecBoot}

Our approach rejects $\Hnull$ in favor of $\Halt$ if at least one of statistics $A_n$ exceeds the corresponding threshold $\xalphan$ or formally if $\exists n \in \n : A_n	 > \xalphan$.

In order to properly choose the thresholds, we define bootstrap statistics $\Anb$ in the following non-standard way. Note, that we cannot use an ordinary scheme with replacement or weighted bootstrap since in a high-dimensional case ($\abs{\Istable} \le p)$ the covariance matrix of bootstrap distribution would be singular which would made inverse covariance matrix estimation procedures meaningless.

First, draw with replacement a sequence $\{\varkappa_i\}_{i=1}^N$ of indices from $\Istable$ and denote $$\Xb_{_i} = X_{\varkappa_i} - \empE{X_j}$$ where $\empE{\cdot}$ stands for averaging over values of index belonging to $\Istable$ e.g., $\empE{X_j} = \frac{1}{\abs{\Istable}}\sum_{j\in \Istable} X_j$.   
Denote the measure $\Xb_i$ are distributed with respect to as $\mathbb{P}^\flat$.
In accordance with \eqref{zijk} define 

\begin{equation}
\Zbiuv := \bootTheta_u \Xb_i\bootTheta_v\Xb_i-\bootTheta_{uv}
\end{equation}
and for technical purposes define

\begin{equation}
\hat Z_{i,uv} := \bootTheta_u X_i\bootTheta_v X_i-\bootTheta_{uv}.
\end{equation}
Now for all central point $t$ define a bootstrap counterpart of $A_n(t)$

\begin{equation} \label{bootAnDef}
\Ab_n(t) \coloneqq \infnorm{\frac{1}{\sqrt{2n}}\inv{S} \overline{\left( \sum_{i \in \Il} \Zbi  - \sum_{i \in \Ir} \Zbi \right)}}
\end{equation}
which is intuitively reasonable due to expansion \eqref{Texpantion}.
And finally we define the bootstrap counterpart of $\An$ as 

\begin{equation}
\Anb = \max_{\substack{t \in \Tn}}  \Anb(t).
\end{equation}
Now for each given $\x \in [0,1]$ we can define quantile functions $\zb_n(\x)$ such that
\begin{equation}\label{tailfunctionDEF}
\zb_n(\x) \coloneqq \inf \left\{z :  \Probboot{\Anb > z} \le\x\right\} .
\end{equation}
Next for a given significance level $\alpha$ we apply multiplicity correction choosing $\alphacorrected$ as 
\begin{equation}\label{alphastardef}
\alphacorrected \coloneqq \sup \left\{\x : \Probboot{\exists n \in \n : \Anb > \zb_n(\x)}  \le \alpha\right\}
\end{equation}
and finally choose thresholds as $\xalphan \coloneqq \zb_n(\alphacorrected)$. 

\begin{remark}
	In order to detect multiple breaks we suggest to repeat the calibration procedure after each detected break using a portion of data acquired after the break. 
\end{remark}

\begin{remark}
	One can choose $\Istable = 1,2, ..., N$ and use the whole given sample for calibration as well as for detection. In fact, it would improve the bounds in \ref{main} and \ref{thSensPrec}, since it effectively means $s=N$. However, in practise such a decision might lead to reduction of sensitivity due to overestimation of the thresholds.
\end{remark}

\subsection{Change-point localization} \label{locsec}

In order to localize a change-point we have to assume that $\Istable \subseteq 1..\tau$. Consider the narrowest window detecting a change-point as $\hat n$:
\begin{equation}\label{narrowestdef}
\hat n \coloneqq \min \left\{n \in \n : \An > \xalphan \right\}
\end{equation}
and the central point where this window detects a break for the first time as 
\begin{equation}\label{tauhatDef}
\hat \tau \coloneqq \min \left\{ t \in \mathbb{T}_{\hat n} : A_{\hat n}(t) > x^\flat_{\hat n}(\alpha) \right\}.
\end{equation}
By construction of the family of test statistics we conclude (up to the confidence level $\alpha$) that the change-point $\tau$ is localized in the interval

\begin{equation}
\left[\hat \tau - \hat n ; \hat \tau + \hat n - 1\right].
\end{equation}
Clearly, if a non-multiscale version of the approach is employed, i.e. $\abs{\n} = \{n\}$,  $n = \hat n$ and the precision of localization (delay of the detection in online setting) equals $n$.

\section{Bootstrap validity} \label{mainressec}
This section states and discusses the theoretical result demonstrating the validity of the proposed bootstrap scheme i.e. 
\begin{equation}
\Prob{\forall n \in \n : \An \le \xalphan} \approx 1-\alpha .
\end{equation} 

Our theoretical results require the tails of the underlying distributions to be light. Specifically,
we impose Sub-Gaussianity vector condition.
 \begin{assumption}[Sub-Gaussianity vector condition]
 	\label{subGaussianVector}
 	\begin{equation}
 	\exists L : \forall i \in 1..N \sup_{\substack{a \in \R^p \\ \normtwo{a} \le 1 }}\E{\exp{\left(\left(\frac{a^T	X_i}{L}\right)^2\right)}} \le 2.
 	\end{equation}
 \end{assumption} 

Naturally, in order to establish a theoretical result we have to assume that a method featuring theoretical guaranties was used for estimating the precision matrices. Such methods include graphical lasso \cite{ravikumar2011}, adaptive graphical lasso \cite{zou2006} and  thresholded de-sparsified estimator based on node-wise procedure \cite{janamb}. These approaches overcome  the high dimensionality of the problem by imposing a sparsity assumption, specifically bounding the maximum number of non-zero elements in a row: $d \coloneqq \max_i\abs{\{j | \trueTheta_{ij} \neq 0 \}}$. These approaches are guaranteed to yield a root-$n$ consistent estimate revealing the sparsity pattern of the precision matrix \cite{ravikumar2011, mypaper, janamb} or formally

\begin{definition} \label{consdef}
	Consider an i.i.d. sample $x_1, x_2, ... x_n \in \R^p$. Denote their precision matrix as $\trueTheta = \inv{\E{X_1X_1^T}}$. Let $p$ and $d$ grow with $n$. A positive-definite matrix $\hatTheta^n$ is a consistent estimator of the high-dimensional precision matrix if 
	\begin{equation}
	\infnorm{\trueTheta - \hatTheta^n}  = O_p\left(\sqrt{\frac{\log p}{n}}\right)
	\end{equation}
	and 
	\begin{equation}
	\forall i,j \in 1..p \text{ and } \trueTheta_{ij}=0 \Rightarrow \hatTheta^{n}_{ij}=0.
	\end{equation}
\end{definition}

Graphical lasso and its adaptive versions impose an assumption, common for $\ell_1$-penalized approaches. 

\begin{assumption}[Irrepresentability condition]\label{irrepass}
	Denote an active set 
	\begin{equation}
	\Sa \coloneqq \left\{(i,j) \in 1..p \times 1..p : \trueTheta_{ij} \neq 0 \right\}
	\end{equation}
	and define a $p^2 \times p^2$ matrix $\Gamma^*\coloneqq \trueTheta \otimes \trueTheta$ where $\otimes$ denotes Kronecker product. Irrepresentability condition holds if there exists $\psi \in (0,1]$ such that 
	\begin{equation}
	\max_{e \notin \Sa} \onenorm{\Gamma^*_{e\Sa} \inv{(\Gamma^*_{\Sa\Sa})} } \le 1-\psi.
	\end{equation}
	
\end{assumption}

The interpretation of irrepresentability condition  under normality assumption is given in \cite{sara} \cite{ravikumar2011}. Particularly, \ref{irrepass} requires low correlation between the elements of empirical covariance matrix from the active set $\Sa$ and from its complement. The higher the constant $\psi$ is, the stricter upper bound is assumed.

While the \ref{irrepass} allows for the recovery of the active set, the test statistic does not explicitly require it, since it is based on the de-sparsified estimators. But this assumption is still essential for demonstrating consistency of estimation of the non-zero elements $\trueTheta_\Sa$ by either graphical lasso or its adaptive versions. Alternatively, one can use thresholded de-sparsified estimator, for which the theoretical guaranties can be established in the absence of \ref{irrepass}.

These observations give rise to the two following assumptions.
 
 \edef\oldassumption{\the\numexpr\value{assumption}+1}
 
 \setcounter{assumption}{0}
 \renewcommand{\theassumption}{Assumption 3.\oldassumption.\Alph{assumption}}
\begin{assumption} \label{thetaHatConsistency}
	Suppose, either graphical lasso or its adaptive version was used with regularization parameter $\lambda_n \asymp \sqrt{{\log p}/{n}}$ and also impose \ref{irrepass}.
\end{assumption}

\begin{assumption}\label{thetaHatConsistency1}
	Suppose, thresholded de-sparsified estimator based on node-wise procedure was used with regularization parameter $\lambda_n \asymp \sqrt{{\log p}/{n}}$.
\end{assumption}

\renewcommand{\theassumption}{Assumption \thesection.\arabic{assumption}}

Now we are ready to establish a result which guarantees that the suggested bootstrap procedure yields proper thresholds.

\begin{theorem}\label{main}
	Assume $\Hnull$ holds and furthermore, let $X_1, X_2,...X_N \in \R^p$ be i.i.d. Let \ref{subGaussianVector} and either \ref{thetaHatConsistency} or \ref{thetaHatConsistency1} hold. 
	Also assume, the spectrum of $\trueTheta$ is bounded. Allow the maximal number $d$ of non-zero elements in a row of the matrix  $\trueTheta$, the size $s$ of the set $\Istable$, the dimensionality $p$, the number $\abs{\n}$ of window sizes being considered, the maximum and minimum window sizes $\nmax$ and $\nmin$ grow with the sample size $N$. Further let $N > 2\nmax$,  $\nmax \ge \nmin$ and also impose the sparsity assumption
	\begin{equation}
	d = o \left(\frac{\sqrt[4]{\min \left\{s, \nmin^2\right\}}}{\abs{\n}^3 \log^{10}(pN)}\right).
	\end{equation}
	Then
	
	\begin{equation}
	\begin{split}
	\biggr\rvert\Prob{\forall n \in \n : \An \le \xalphan} &- (1-\alpha)\biggr\rvert
	= o_P(1).
	\end{split}
	\end{equation}
\end{theorem}
The finite-sample version of this result, namely, \ref{mainFS}, is given in Appendix \ref{mainProof} along with the proofs.

\paragraph{Bootstrap validity result discussion}
\ref{main} guarantees under mild assumptions (\ref{irrepass} seems to be the most restrictive one, yet it may be dropped if the node-wise procedure is employed) that the first-type error rate meets the nominal level $\alpha$ if the narrowest window size $\nmin$ and the set $\Istable$ are large enough. Clearly, the dependence on dimensionality $p$ is logarithmic which establishes applicability of the approach in a high-dimensional setting.
It is worth noticing that, unusually, the sparsity bound gets stricter with $N$ but the dependence is only logarithmic. Indeed, we gain nothing from longer samples, since we use only $2n$ data points each time.

\paragraph{On-line setting}\label{online}
As one can easily see, the theoretical result is stated in off-line setting, when the whole sample of size $N$ is acquired in advance. In on-line setting we suggest to control the probability $\alpha$  to raise a false alarm for at least one central point $t$ among $N$ data points (which differs from  the classical techniques controlling the mean distance between false alarms \cite{aries2007optimal}). Having $\alpha$ and $N$ chosen one should acquire $s$ data-points (the set $\Istable$), use the proposed bootstrap scheme with bootstrap samples of length $N$ in order to obtain the thresholds. Next the approach can be naturally applied in on-line setting and \ref{main} guarantees the capability of the proposed bootstrap scheme to control the aforementioned probability to raise a false alarm.

\paragraph{Proof discussion}
The proof of the bootstrap validity result, presented in Appendix \ref{mainProof}, mostly relies on the high-dimensional central limit theorems obtained in \cite{Chernozhukov2014}, \cite{chernozhukov2013}. These papers also present bootstrap justification results, yet do not include a comprehensive bootstrap validity result. The theoretical treatment is complicated by the randomness of $\xalphan$. We overcome it by applying the so-called ``sandwiching'' proof technique (see \ref{sandwitch}), initially used in \cite{SpokWillrich}. The high-level structure of the proof can be summarized as follows:

\begin{enumerate}
	\item Approximate statistics $\An$ by norms of sub-vectors $\eta_n$ of a high-dimensional Gaussian vector $\eta$ up to the residual $\Ra$ using the high dimensional central limit theorem by \cite{Chernozhukov2014}. 
	\item Similarly, approximate bootstrap counterparts $\Anb$ of the statistics by norms of sub-vectors $\zeta_n$ of a high-dimensional Gaussian vector $\zeta$ up to the residual $\Rboot$.
	\item Prove that the covariance matrix $\Var{\zeta}$ is concentrated in the ball of radius $\deltaY$ centered at its real-world counterpart $\Var{\eta}$.
	\item By employing the Gaussian comparison result provided by \cite{Chernozhukov2014} and \cite{chernozhukov2013} obtain similarity of joint distributions of the norms of $\zeta$ and $\eta$, which in combination with the steps 1 and 2 yields similarity of joint distributions of $\An$ and $\Anb$.
	\item Finally, obtain the bootstrap validity result using the sandwiching \ref{sandwitch}.
\end{enumerate}

The rigorous treatment of steps 1 through 4 is presented in Sections \ref{GAprA}, \ref{GaprAB}, \ref{twomatsec} and \ref{AAbSim} respectively, while step 5 is formalized in Sections \ref{mainProof} and \ref{appa}.

\section{Sensitivity and Consistency results} \label{precsensitivitysec} 
Consider the following setting. Let there be index $\tau$, such that $\{X_i\}_{i \le \tau}$ are i.i.d. and $\{X_i\}_{i > \tau}$ are i.i.d. as well. Denote precision matrices $\inv{\Theta_1} \coloneqq \E{X_1X_1^T}$ and $\inv{\Theta_2} \coloneqq \E{X_{\tau+1}X_{\tau+1}^T}$.   Define the break extent $\Delta$ as  
\begin{equation} \label{breakdef}
\Delta \coloneqq \infnorm{{\Theta_1 - \Theta_2}}.
\end{equation}
The question is, how large the window size $n$ should be in order to reliably reject $\Hnull$ and how firmly can we localize the change-point. 

\begin{theorem} \label{thSensPrec}
	Let \ref{subGaussianVector} and either \ref{thetaHatConsistency} or \ref{thetaHatConsistency1} hold. 
	Also assume, the spectrums of $\Theta_1$ and $\Theta_2$ are bounded. Allow the maximal number $d$ of non-zero elements in a row of the matrix  $\trueTheta$, the size $s$ of the set $\Istable$, the dimensionality $p$, the number $\abs{\n}$ of window sizes being considered, the maximum and minimum window sizes $\nmax$ and $\nmin$ grow with the sample size $N$ and let the break extent $\Delta$ decay with $N$. Further let $N > 2\nmax$,  $\nmax \ge \nmin$, 
	\begin{equation}\label{nminbound}
	d = o \left(\frac{\sqrt{\max \{ s,\nmin\}}}{\abs{\n}\log^{7}(pN)}\right)
	\end{equation}
	and there exists $n^* \in \n$ such that 
	\begin{equation}\label{nsuffbound}
	\frac{\log^2(pN)  }{n^* \Delta} = o(1).
	\end{equation}
	Then $\Hnull$ will be rejected with probability approaching $1$.
\end{theorem}

This result is a direct corollary of the finite-sample sensitivity result established and discussed in Appendix \ref{sensitivityPrecisionMatrixResult}.

The assumption $\Istable \subseteq 1..\tau$ is only technical. 
A similar result may be proven without relying on it by methodologically the same argument.   See \ref{dropass} for more details. 

Next we formulate a trivial corollary of \ref{thSensPrec} establishing consistency of the change-point estimator $\hat \tau $ defined by \eqref{tauhatDef}. 

\begin{corollary}\label{consistency}
	Let assumptions of \ref{thSensPrec} hold. Moreover, assume a suitable window is being considered 
	
	\begin{equation}\label{nsuffdef}
	n^* \coloneqq \left\lceil \left(\frac{\log^2(pN)  }{\Delta}\right)^\eta \right\rceil  \in \n, \text{ where } \eta > 1.
	\end{equation}
	then 
	\begin{equation}\label{key}
	\Prob{\abs{\tau - \hat \tau} \le n^*} \gtrsim
	1 - \alpha.
	\end{equation}
\end{corollary}

\paragraph{Sensitivity and consistency results discussion}

Assumptions \eqref{nminbound} and \eqref{nsuffbound} are essentially a sparsity bound and a bound for the sufficient windows size $n^*$. Clearly,  they do not yield a particular value $n^*$ necessary to detect a break, since it depends on the underlying distributions. A more restrictive assumption \eqref{nsuffdef} suggests a particular choice of $n^*$ which is suitable in asymptotic setting. Note, the results include dimensionality $p$ only under the sign of logarithm, which guarantees high sensitivity of the test and proper localization of the change-point in high-dimensional setting.

\paragraph{Online setting}
\ref{thSensPrec} and \ref{consistency} are established in off-line setting. In on-line setting they guarantee that the proposed approach can reliably detect a break of an extent not less than $\Delta$ with a delay at most $n^*$.

\paragraph{Multiple change-point detection}
Consider a setting which allows for multiple breaks: let there be $n_b$ change-points $\{\tau_j\}_{j=1}^{n_b} \subset N$ and let $\{X_i\}_{i=\tau_{j}+1}^{\tau_{j+1}}$ be i.i.d for all $1 \le j \le n_b$. Also define $n_b+1$ precision matrices $\{\Theta_j\}_{j=0}^{n_b}$ -- one for each region between a pair of consecutive change-points. Define the minimal extent of the breaks as 

\begin{equation}
\Delta_{\min} = \min_{j} \infnorm{\Theta_j - \Theta_{j+1}}.
\end{equation}
Then under assumption that two change-points are not too close to each other 
\begin{equation} 
\min_j \abs{\tau_j - \tau_{j+1}}  \ge n^* + s
\end{equation}
we can apply \ref{thSensPrec} and conclude that the method detects all the change points with probability approaching $1$.

\begin{remark} \label{breakextentremark}
	The sensitivity and consistency results heavily depend on the break extent $\Delta$ defined by \eqref{breakdef}. As \ref{remarkOnCovStat} suggests, a test statistic based on the $\ell_\infty$ norms of the covariance matrices, not of their inverses, may also be considered. For such a statistic a similar result could be established, but the break extent would have to be correspondingly redefined as a $\ell_\infty$ distance between the covariance matrices instead of $\Theta_1$ and $\Theta_2$ which would unnecessarily restrict the application of this approach in the field of neuroimaging where the precision matrix is the main object of interest. 
\end{remark}

\section{Simulation study} \label{simsec}
\subsection{Design}\label{secdes}

In our simulation we test $$\Hnull = \left\{\{X_i\}^N_{i=1}  \sim \N{0}{I} \right\} $$ versus an alternative 

$$\Halt = \left\{ \exists  \tau : \{ X_i \}_{i=1}^{\tau} \sim \N{0}{I} \text{ and } \{ X_i \}_{i=\tau+1}^{N} \sim \N{0}{\Sigma_{1}} \right\}.$$
The alternative covariance matrix $\Sigma_1$ was generated in the following way. First we draw $k \sim Poiss(3)$. The matrix $\Sigma_1$ is composed as a block-diagonal matrix of $k$ matrices of size $2\times2$ with ones on their diagonals and their off-diagonal element drawn uniformly from $[-0.6;-0.3] \cup [0.3; 0.6]$ and an identity matrix of size $(p-2k) \times (p-2k)$. 
The dimensionality of the problem is chosen as $p=50$, the length of the sample $N=1000$ and we choose the set $\Istable = [1, 2 ,.. 100]$. The absence of positive effect of large sample size $N$ is discussed in Sections \ref{mainressec} and \ref{precsensitivitysec}. Moreover, in all the simulations under alternative the sample was generated with the change point in the middle: $\tau=N/2$ but the algorithm was oblivious about this as well as about either of the covariance matrices. The significance level $\alpha = 0.05$ was chosen.
In all the experiments graphical lasso with penalization parameter $\lambda_n = \sqrt{\frac{\log p}{n}}$ was used in order to obtain $\hatTheta^\s_n(t)$. In the same way, graphical lasso with  penalization parameter $\lambda_s$ was used in order to obtain $\hatTheta$. 

We assess the performance of the proposed approach in both on-line and off-line settings. Multiple break detection is left out of the scope because the suggested method attacks the problem repetitively detecting the breaks one-by-one in on-line fashion. In on-line setting the method is calibrated in order to raise a false alarm with probability $\alpha = 0.05$ on $N$ data points using for calibration the set $\{X_i\}_{i \in \Istable}$  which is known in advance (unlike the rest of the dataset).

\subsection{Experiment results}

We have also come up with an approach to the same problem not involving bootstrap. The paper \cite{twoCov} defines a high-dimensional two-sample test for equality of matrices. Moreover, the authors prove asymptotic normality of their statistic which makes computing p-value possible. 
We suggest to run this test for every $t \in \Tn$ and every $n \in \n$, adjust the obtained p-values using Holm method \cite{holm} and eventually compare them against $\alpha$. This example is considered in order to demonstrate inapplicability of a straightforward application of a two-sample test for change-point detection.

The paper \cite{Matteson2015} suggests an approach based on comparing characteristic functions of random variables. The critical values were chosen with permutation test as proposed by the authors. In our experiments the method was allowed to consider all the sample at once. The R-package ecp \cite{ecppackage} was used.

In \cite{Cho2015} a high-dimensional approach aiming to detect a change-point in second-order structure of the time series is suggested. In order to investigate its performance in our setting we use the implementation kindly provided by the authors of the paper. 

Our approach is implemented as an R-package covcp, which is available on GitHub\footnote{\url{https://github.com/akopich/covcp}}.

Below we report and discuss results for a particular alternative matrix $\Sigma_1$ generated in such a way suggested in Section \ref{secdes}. All the methods being considered exhibit comparable performance for other matrices $\Sigma_1$ drawn from the same distributions. 

The first type error rate and power for our approach  are reported in Table~\ref{table}. 
As one can see, our approach allows to properly control first type error rate in both off-line and on-line setting. In fact, the test is conservative and we believe this is caused by the $\le$ signs entering the definitions of $\zb_n(\cdot)$ \eqref{tailfunctionDEF} and the multiplicity correction \eqref{alphastardef}. As expected, the power of the test is higher for larger windows and it is decreased by adding narrower windows into consideration which is the price to be paid for better localization of a change point. The power of the test is rather similar in on-line and off-line settings which is due to its local nature.

In our study the approaches proposed in \cite{Matteson2015}, \cite{Cho2015} and the one based on the two sample test \cite{twoCov} turned out to be conservative, but neither of them exhibited power above $0.1$.

In order to justify application of multiscale approach (i. e. $\abs{\n} > 1$) for the sake of better change-point localization  in off-line setting we report the distribution of the narrowest detecting window $\hat n$ (defined by \eqref{narrowestdef}) over $\n$ in Figure \ref{fig:pecharts}. The Table~\ref{table} represents average precision of change-point localization for various choices of set of window sizes $\n$. One can see, that multiscale approach significantly improves the precision of localization. 

Similarly, in on-line setting we analyze {\it delay of detection} -- the average number of data points after the break which have to be considered before the method detects a break. As Table~\ref{table} shows, the delay of detection is significantly decreased by the use of multiscale approach which justifies its use in on-line setting as well.
%Bringing a narrow window with low chance of detection of the break on its own ($n=70$ in our experiment, which demonstrates the power of only $0.08$ if $\n=\{70\}$) into consideration may somewhat decrease precision of localization (or correspondingly increase delay of detection) which is the price of multiplicity correction. Generally, the window being included should contribute enough to the detection process in order to overwhelm the decrease in sensitivity incurred by the increased number of statistics. 

\begin{table}
	\centering
	\caption{		First type error rate and power of change-point localization of the proposed approach for various sets of window sizes $\n$. Localization precision and delay of detection are also reported for off-line and on-line settings respectively. 	}\label{table}
	\begin{tabular}{|c|c|c|c|c|c|c|} 
		\hline
		\multicolumn{1}{|l|}{} & \multicolumn{3}{c|}{Off-line}                                                                                                            & \multicolumn{3}{c|}{On-line}                                                                                                         \\ 
		\hline
		$\n$                   & \begin{tabular}[c]{@{}c@{}}I type \\error rate \end{tabular} & Power & \begin{tabular}[c]{@{}c@{}}Localization \\precision \end{tabular} & \begin{tabular}[c]{@{}c@{}}I type \\error rate \end{tabular} & Power & \begin{tabular}[c]{@{}c@{}}Delay of \\detection\end{tabular}  \\ 
		\hline
		\{70\}                  & 0.02                                                         & 0.09  & 70                                                                & 0.02                                                         & 0.08  & 46                                                            \\ 
		\hline
		\{100\}                 & 0.00                                                         & 0.37  & 100                                                               & 0.00                                                         & 0.37  &  97                                                           \\ 
		\hline
		\{140\}                 & 0.01                                                         & 0.81  & 140                                                               & 0.01                                                         & 0.81  &  127                                                           \\ 
		\hline
		\{200\} & 0.00   & 0.99  & 200  & 0.00   & 0.99  & 160      \\ 
		\hline
		\{140, 200\} & 0.00   & 0.99  & 153  & 0.00 & 0.99  & 140      \\ 
		\hline
		\{100, 140, 200\} & 0.00   & 0.98  & 146  & 0.00    & 0.99  & 134      \\ 
		%\hline
		%\{70, 100, 140, 200\} & 0.00   & 0.98  & 148  & 0.01    & 0.97  &  138      \\ 
		\hline
		\{70, 140\}             & 0.01                                                         & 0.76  & 135                                                               & 0.01                                                         & 0.76  & 124                                                          \\ 
		\hline
		\{100, 140\}            & 0.01                                                         & 0.75  & 124                                                               & 0.01                                                         & 0.75  &  118                                                           \\ 
		\hline
		\{70, 100, 140\}        & 0.01                                                         & 0.74  & 123                                                               & 0.01                                                         & 0.72  &  118                                                           \\
		\hline
	\end{tabular}
\end{table}

\begin{figure}
	\caption{{Pie charts representing distribution of narrowest detecting window $\hat n$ and the precision of localization in cases of $\abs{\n} = \{70, 140\}$, $\abs{\n} = \{100, 140\}$ and $\abs{\n} = \{70,100, 140\}$ respectively}}
	\label{fig:pecharts}
	\centering
	\includegraphics[width=\linewidth]{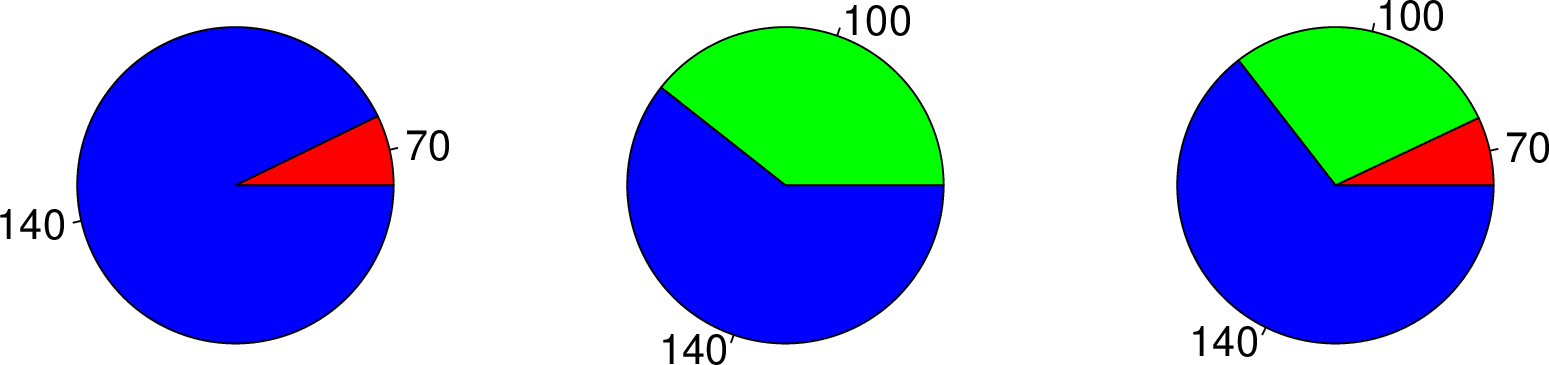}
\end{figure}

\section{Analysis of real-world datasets} \label{real}
The paper \cite{Puschmann} presents a functional magnetic resonance
imaging study on human subjects, who had to learn the relation between different auditory stimuli and a monetary reward while being scanned. Based on their recorded performance (e.g. success rate) those of them who have managed to learn the relation within the time course of the experiment are considered {\it learners}. The authors have investigated the activity of $4$ brain regions-of-interest (ROI) which are believed to be involved in solving problems like the suggested one. Among the learners $3$ of the ROIs exhibit a statistically significant change in the Blood Oxygenation Level Dependent (BOLD) response between the first and the last quarters of the experiment. 

In this paper we analyze the BOLD responses for $18$ subjects which were classified as learners. We have considered a finer-grained brain atlas with $p=256$ ROIs \cite{Finn2015}. For each subject and for each ROI a time-course of length $N=1680$ was acquired. Datasets for each subject were analyzed separately. We have applied the proposed approach to the residuals of linear modeling usual in fMRI experiments \cite{Poldrack2011}. The residuals are publicly available\footnote{\url{http://www.wias-berlin.de/preprint/2404/wias_preprints_2404.zip}}. We have used graphical lasso with the penalty parameter $\lambda_n = \sqrt{\frac{\log p}{n}}$, a single window size $n = 50$ was considered and the first $200$ data points were used for bootstrap simulations ($\Istable = \{1...200\}$). The proposed approach has detected a change-point in the covariance structure of residuals at significance level $\alpha = 0.05$ for each of the subjects.

\appendix

\section{Proof of sensitivity result} \label{sensitivityPrecisionMatrixResult}
\begin{proof}[Proof of \ref{thSensPrec}]
	Proof consists in applying of the finite-sample \ref{thSensPrecFS}. Its applicability is guaranteed by the consistency results given in papers \cite{ravikumar2011, mypaper, janamb} and by the results from \cite{sara, janamb, mypaper} bounding the term $\Rt$. High probability of set $\Tauarr$ is ensured by \ref{sigmaconcentration}.
\end{proof}
\begin{theorem} \label{thSensPrecFS} Let $\Istable \subseteq 1..\tau$.
	Let $\hatTheta$ denote a symmetric estimator of $\Theta_1$ s.t. for some $r \in \R$ it holds that 
	
	\begin{equation}
	\infnorm{\Theta_1 - \hatTheta} < r
	\end{equation}
	and $(\Theta_1)_{ij} = 0 \Rightarrow \hatTheta_{ij}=0$. 	
	Suppose \ref{subGaussianVector} holds and there exists $\Rt$ such that $\infnorm{r_{\nmax}^\s(t)} \le \Rt$ for all $\s\in \{l,r\}$ and $t \in \mathbb{T}_{\nmax}$ on some set 
	\begin{equation}
	\begin{split}
	\Tauarr &\coloneqq  \left\{ \forall t \le \tau - \nmax : \infnorm{\hatSigma_n^\s (t) - \trueSigma_1}\le \delta_{\nmax} \right\} \\&\bigcap\left\{ \forall t \ge \tau + \nmax : \infnorm{\hatSigma_n^\s (t) - \trueSigma_2}\le \delta_{\nmax} \right\}.
	\end{split}
	\end{equation} Moreover, let the residual $\Rboot$ defined in \ref{bootstrapGarLemma} be bounded:
	
	\begin{equation}\label{key}
	\Rboot  \le \frac{\alpha}{6\abs{\n}}.
	\end{equation} 
	Also let 
	\begin{equation}\label{nass}
	\sqrt{\frac{\nmax}{2}}\infnorm{{S}}\left(\Delta - 2\Rt\right) \ge \q,
	\end{equation}
	where 
	\begin{equation}\label{precqdef}
	\q \coloneqq \sqrt{2\left(1+\deltaY(r)\right) \log \left(\frac{2  N\abs{\n} p^2}{\alpha -3\abs{\n}\Rboot}\right)}
	\end{equation}
	and $\deltaY$ is defined in \ref{twomaticies}.
	Then on set $\Tauarr$ with probability at least
	\begin{equation}
	1-\probsigmay,
	\end{equation}
	where $\probsigmay$ is defined in \ref{twomaticies},
	$\Hnull$ will be rejected.
\end{theorem}
\paragraph{Discussion of finite-sample sensitivity result}
The assumption \eqref{nass} is rather complicated. Here we note that if either graphical lasso \cite{ravikumar2011}, adaptive graphical lasso \cite{zou2006} or  thresholded de-sparsified estimator based on node-wise procedure \cite{janamb} with penalization parameter chosen as $\lambda_s \asymp o(\sqrt{\log p / n}) $ was used, given $d,s,p,N,\nmin, \nmax \rightarrow \infty$, $N > 2\nmax$, $\nmax \ge \nmin$, $s \ge \nmin$  and $d = o(\sqrt{\nmax})$ it boils down to
\begin{equation}
\nmax \ge D_6\frac{1}{\Delta} \left(\infnorm{\inv{S}}\log(N\abs{\n}p^2)  \right)^2
\end{equation}
for some positive constant $D_6$ independent of $N, \n, p, d, S$ while the parameters $q$, $\gamma$ and $x$ may be chosen as in \eqref{precq}, \eqref{precgamma}, \eqref{precx} (high probability of $\Tauarr$ is ensured by \ref{sigmaconcentration}).
At the same time the remainder $\Rboot$ can be bounded by \eqref{abootbound}.

As expected, the bound for sufficient window size decreases with growth of the break extent $\Delta$ and the size of the set $\Istable$, but increases with dimensionality $p$. It is worth noticing, that the latter dependence is only logarithmic. And again, in the same way as with  \ref{main}, the bound increases with the sample size $N$ (only logarithmically) since we use only $2n$ data points.

\begin{proof}[Proof of \ref{thSensPrecFS}]
	
	The strategy of the proof is straightforward. 
	\begin{enumerate}
		\item  Bound the probability of a large deviation of the Gaussian approximation $\infnorm{\zeta_{\nmax}}$ of $\Ab_{\nmax}$ using Chernoff bound.
		\item  Bound the corresponding critical level $\xalphann{\nmax}$ up to the remainder term $\Rboot$, using the approximation guaranties provided by  \ref{bootstrapGarLemma}.
		\item  Conclude that $A_{\nmax} > \xalphann{\nmax}$ and therefore the $\Hnull$ is rejected, by construction of the test statistic $A_{\nmax}$.
	\end{enumerate}
	First, denote all the window sizes being considered as $n_1, n_2, ... n_{\abs{\n}}$: 
	
	\begin{equation}\label{key}
	\n = \{n_1, n_2, ...,  n_{\abs{\n}}\},
	\end{equation}
	\begin{equation}\label{key}
	\nmin = n_1 < n_2 < ... < n_{\abs{\n}} = \nmax.
	\end{equation}
	Consider a pair of centered normal vectors
	
	\begin{equation}
	\eta \coloneqq \left(\begin{array}{cccc}
		\eta_{n_1} &\eta_{n_2} &... &\eta_{ n_{\abs{\n}}}
	\end{array}\right) \sim \N{0}{\trueSigmaY},
	\end{equation}
	\begin{equation}
		\zeta \coloneqq \left(\begin{array}{cccc}
			\zeta_{n_1} &\zeta_{n_2} &... &\zeta_{ n_{\abs{\n}}}
		\end{array}\right) \sim \N{0}{\hatSigmaY},
	\end{equation}
	\begin{equation}
		\trueSigmaY \coloneqq\frac{1}{{2\nmax}} \sum_{j=1}^{2\nmax} \Var{Y^{n}_{\cdot j}},
	\end{equation}	
	\begin{equation}
		\hatSigmaY \coloneqq\frac{1}{{2\nmax}} \sum_{j=1}^{2\nmax} \Var{Y^{n\flat}_{\cdot j}},
	\end{equation}	 
	where vectors $Y^n_{\cdot j}$ and $Y^{n\flat}_{\cdot j}$ are defined in proofs of  \ref{SGar} and \ref{SbGarRandomBounds} respectively. 
	
	\ref{maxBound} applies here and yields for all positive $\q$
	\begin{equation}
	\Prob{ \infnorm{\zeta_{\nmax}} \ge \q } \le 2\abs{\Tnmax}p^2\exp\left(-\frac{\q^2 }{2\infnorm{\hatSigmaY} }\right),
	\end{equation}
	where $\hatSigmaY = \Var{\zeta}$ and $\abs{\Tnmax}$ is the number of central points for window of size $\nmax$.  Applying \ref{twomaticies} on a set of probability at least $1-\probsigmay$ yields 
	$\infnorm{\trueSigmaY - \hatSigmaY} \le \deltaY$, and hence, due to the fact that $\infnorm{\trueSigmaY} = 1$ by construction,
	\begin{equation}
	\Prob{ \infnorm{\zeta_{\nmax}} \ge \q } \le 2\abs{\Tnmax} p^2\exp\left(-\frac{\q^2 }{2\left(1 + \deltaY\right)  }\right).
	\end{equation}
	Due to \ref{bootstrapGarLemma} and continuity of Gaussian c.d.f.
	\begin{equation}
	\Probboot{A^\flat_{\nmax} \ge \xalphann{\nmax}} \ge \alpha/\abs{\n} -  2\Rboot
	\end{equation}
	and due to \ref{bootstrapGarLemma} along with the fact that $\abs{\Tnmax} < N$, choosing $\q$ as proposed by equation \eqref{precqdef}
	we ensure that $\xalphann{\nmax} \le\q $.
	
	Now by assumption of the theorem and by construction of the test statistics $\An$
	
	\begin{equation}
	A_{\nmax} \ge \sqrt{\frac{\nmax}{2}}\infnorm{{S}}\left(\Delta - 2\Rt\right).
	\end{equation}
	Finally, we notice that due to assumption \eqref{nass} $A_{\nmax} > \q$ and therefore, $\Hnull$ will be rejected.
\end{proof}

\begin{remark}\label{dropass}
The assumption $\Istable \subseteq 1..\tau$ is only technical. 
A similar result may be proven without relying on it by methodologically the same argument. Really, if we drop the assumption, the matrix $S$ (depending only on the distribution before the break) will fail to normalize the vectors $Y^{n\flat}_{\cdot j}$ and therefore $\infnorm{\hatSigmaY}$ will significantly deviate from $1$. Yet a bound (omitted for brevity) of sort $\infnorm{\hatSigmaY} < C$  where $C$ would depend on $\deltaY$, distributions before and after the break as well as on the portions of the data points before and after the break included in the $\Istable$ can be established. The term $C$ will have to enter the definition \eqref{precqdef} of $\q$ instead of $(1+\deltaY)$.
\end{remark}

\begin{lemma}
	Consider a centered random Gaussian vector $\xi\in \R^p$ with arbitrary covariance matrix $\Sigma$. 
	For any positive $\q$ it holds that 
	\begin{equation}
	\Prob{\max_i \xi_i \ge \q } \le p\exp\left(-\frac{\q^2 }{2\infnorm{\Sigma} }\right).
	\end{equation}
\end{lemma}

\begin{proof}
	By convexity we obtain the following chain of inequalities for any $t$
	\begin{equation}
	\ex{t\E{\max_i \xi_i}} \le \E{\ex{t\max_i \xi_i}} \le \E{\ex{t\sum_i \xi_i}} \le p \ex{t^2\infnorm{\Sigma}/2}.
	\end{equation}
	Chernoff bound yields for any $t$ 
	\begin{equation}
		\Prob{\max_i \xi_i \ge \q } \le  \frac{p \ex{t^2\infnorm{\Sigma}/2}}{\ex{t\q}}.
	\end{equation}	
	Finally, optimization over $t$ yields the claim.
\end{proof}

As a trivial corollary, one obtains 
\begin{lemma}\label{maxBound}
	Consider a centered random Gaussian vector $\xi\in \R^p$ with arbitrary covariance matrix $\Sigma$. 
	For any positive $\q$ it holds that 
	\begin{equation}
	\Prob{ \infnorm{\xi} \ge \q } \le 2p\exp\left(-\frac{\q^2 }{2\infnorm{\Sigma} }\right).
	\end{equation}
\end{lemma}

\section{Proof of bootstrap validity result}\label{mainProof}
\begin{proof}[Proof of \ref{main}]
	Proof consists in applying of the finite-sample \ref{mainFS}. Its applicability is guaranteed by the consistency results given in papers \cite{ravikumar2011, mypaper, janamb} and by the results from \cite{sara, janamb, mypaper} bounding the term $\Rt$. High probability of set $\TauT$ is ensured by \ref{sigmaconcentration}.
\end{proof}

\begin{theorem}\label{mainFS}
	Assume $\Hnull$ holds and furthermore, let $X_1, X_2,...X_N$ be i.i.d.
	Let $\hatTheta$ denote a symmetric estimator of $\trueTheta$ s.t. for some positive $r$
	
	\begin{equation}
	\infnorm{\trueTheta - \hatTheta} < r
	\end{equation}
	and $\trueTheta_{ij} = 0 \Rightarrow \hatTheta_{ij}=0$. 	
	Suppose \ref{subGaussianVector} holds and there exists $\Rt$ such that $\sqrt{n}\infnorm{r_n^\s(t)} \le \Rt$ for all $\s\in \{l,r\}$, $n \in \n$ and $t\in \Tn$ on set 
	\begin{equation}
	\TauT \coloneqq \left\{\forall \s \in \{l,r\}, n \in \n, t \in\Tn : \infnorm{\hatSigma_n^\s (t) - \E{X_1 X_1^T}} \le \delta_n \right\}.
	\end{equation}
	 Moreover, let 
	
	\begin{equation}\label{key}
	R \coloneqq \left(3+2\abs{\n}\right)\left(2\Ra(\Rt) + 2\Rboot + \Rsigma(r)\right) \le \frac{\alpha}{2},
	\end{equation}
	where the remainders $\Ra$, $\Rboot$, $\Rsigma$ are defined in \ref{AGar}, \ref{bootstrapGarLemma} and \ref{sandwitch} respectively and the mis-tie $\deltaY$ involved in the definition of $\Rsigma$ comes from \ref{twomaticies}.
	Then on set $\TauT$ it holds that 
	\begin{equation}
	\abs{\Prob{\forall n \in \n : \An \le \xalphan} - (1-\alpha ) }  \le R	+ 2(1-q).
	\end{equation}
	where 
	\begin{equation}\label{probability}
	q =1-\probsigmay-\probSigmaBound{s}{\gamma} - \probMoment{x}{t} 
	\end{equation}
	and the terms $\probsigmay$, $\probSigmaBound{s}{\gamma}$ and $\probMoment{x}{t}$ are defined in \ref{twomaticies}, \ref{sigmaconcentration} and \ref{bootstrapGarLemma} respectively.
\end{theorem}

\paragraph{Discussion of finite-sample bootstrap validity result}\label{remark}
The terms $\deltaY$, $\Ra$, $\Rboot$ and $\Rsigma$ involved in the statement of \ref{mainFS} are rather complicated. The exact expressions for them  are provided by  \ref{twomaticies}, \ref{AGar}, \ref{bootstrapGarLemma} and \ref{sandwitch} respectively, 3rd and 4th moments $M_3^3$ and $M_4^4$ involved therein are bounded by \ref{thirdAndFourthMomentBounds} and  \ref{ZconcentrationComponentWise} while asymptotic bounds for $\Rt$ are provided in \cite{janamb} (for node-wise procedure) and \cite{sara} (for graphical lasso). For the case of graphical lasso an explicit form of $\Rt$ is given in \cite{mypaper}.

Here we just note that if $\hatTheta$ is a root-$n$ consistent estimator, recovering sparsity pattern (graphical lasso \cite{ravikumar2011}, adaptive graphical lasso \cite{zou2006} or  thresholded de-sparsified estimator based on node-wise procedure \cite{janamb}), then for $d,s,p,N,\nmin, \nmax \rightarrow \infty$, $N > 2\nmax$, $\nmax \ge \nmin$, $s \ge \nmin$ and $\frac{d^2}{\nmin} = o(1)$ given the spectrum of $\trueTheta$ is bounded 

\begin{equation}\label{abootbound}
\Rboot \le D_1  \left(\frac{ L^4d\log^{7}(2p^2T\nmax)}{\nmin}\right)^{1/6}\log^2(ps).
\end{equation}
If either  graphical lasso, adaptive graphical lasso or node-wise procedure \cite{meinshausen2006} is used with $\lambda_n \asymp \sqrt{\frac{\log p}{n}}$ in order to obtain $\hatTheta^\s_n(t)$, then on set $\TauT$ it holds that

\begin{equation}\label{raass}
\Ra \le D_2\left(\frac{L^4d \log^{7}(2p^2T\nmax)}{\nmin}\right)^{1/6} +  D_3 \sqrt{\frac{{\log 2p^2T}}{{\nmin}}} d \log p .
\end{equation}
The high probability of $\TauT$ may be ensured by means of \ref{sigmaconcentration} e.g., choosing $\gamma = \log (500T)$ for $\Prob{\TauT} \ge 0.99$.
Further
\begin{equation}
\deltaY \le D_4   \frac{L^4d^2}{\sqrt{s}},
\end{equation}

\begin{equation}
\Rsigma \le
D_5\left(\frac{L^4d^2}{\sqrt{s}}\right)^{1/3} \log^{2/3}(2p^2T).
\end{equation}
Here $D_1, ..., D_5$ are positive constants independent of $N$, $\n$, $d$, $p$ and $s$.
We also note that the proper choice of $x$, $\gamma$ and $q$ in \eqref{probability} is 

\begin{equation}\label{precx}
x = 6,
\end{equation}

\begin{equation}\label{precgamma}
\gamma = \log(500T),
\end{equation}

\begin{equation}\label{precq}
q = 7+ 4\log(p)
\end{equation}
which ensures the probability defined by \eqref{probability} to be above $0.99$. For exact expression of $\probsigmay$, $\probSigmaBound{s}{\gamma}$ and $\probMoment{x}{t}$ see \ref{twomaticies}, \ref{sigmaconcentration} and \ref{bootstrapGarLemma}.

\begin{proof}[Proof of \ref{mainFS}]
	The proof consists in application of \ref{SbGarRandomBounds}, \ref{SGar} and \ref{precTVlemma} justifying applicability of \ref{sandwitch}.
\end{proof}

\section{Sandwiching lemma} \label{appa}
\begin{lemma}\label{sandwitch}
	Consider a normal multivariate vector $\eta$ with a deterministic covariance matrix and  a normal multivariate vector $\zeta$ with a possibly random covariance matrix such that
	\begin{equation}\label{gar}
	\sup_{\{x_n\}_{n \in \n} \subset \R} \abs{\Prob{\forall n \in \n : \An \le x_n} - \Prob{\forall n \in \n  : \infnorm{\eta_n} \le x_n}}  \le \Ra ,
	\end{equation}
	
	\begin{equation} \label{garb}
	\sup_{\{x_n\}_{n \in \n} \subset \R} \abs{\Probboot{\forall n \in \n : \Anb \le x_n} - \Probboot{\forall n \in \n  : \infnorm{\zeta_{n}} \le x_n}}  \le \Rboot, 
	\end{equation}
	
	\begin{equation} \label{tv}
	\sup_{\{x_n\}_{n \in \n} \subset \R} \abs{\Prob{\forall n \in \n : \An \le x_n} - \Probboot{\forall n \in \n : \Anb \le x_n}} \le R.
	\end{equation}
	where $\eta_n$ and $\zeta_n$ are sub-vectors of $\eta$ and $\zeta$ respectively.
	Then 
	
	\begin{equation}
	\abs{\Prob{\forall n \in \n : \An \le \xalphan} - (1-\alpha)} \le \left(3+2\abs{\n}\right) \left(R +\Ra + \Rboot\right).
	\end{equation}
\end{lemma}

\begin{proof}
	Let us introduce some notation. Denote multivariate cumulative distribution function of $\An, \Anb, \infnorm{\eta_{n}}, \infnorm{\zeta_{n}}$ as   $P, \Pb, \NN, \Nb : \R^{\abs{\n}} \rightarrow [0,1]$ respectively.	
	Define the following sets for all $\delta\in [0,\alpha]$
	
	\begin{equation} \label{zmdef}
	\Zp(\delta) \coloneqq  \left\{z : \NN(z) \ge 1-\alpha -\delta\right\},
	\end{equation}
	
	\begin{equation}
	\Zm(\delta) \coloneqq \left\{z : \NN(z) \le  1-\alpha + \delta\right\}
	\end{equation}
	and their boundaries 
	
	\begin{equation} \label{bzmdef}
	\bZp(\delta) \coloneqq  \left\{z : \NN(z) = 1-\alpha - \delta \right\},
	\end{equation}
	
	\begin{equation}
	\bZm(\delta) \coloneqq \left\{z : \NN(z) =  1-\alpha + \delta \right\}.
	\end{equation}
	Consider $\delta = R + \Ra + \Rboot $ and denote sets $\Zp = \Zp(\delta)\text{, } \Zm = \Zm(\delta)\text{, }  \bZm = \bZm(\delta)\text{, } \bZp = \bZp(\delta)$
	Define  a set of thresholds satisfying the confidence level 
	\begin{equation}
	\Zcb \coloneqq \left\{z : \Pb(z) \ge 1-\alpha ~\&~\forall z_1 < z :  \Pb(z_1) < 1-\alpha\right\} 
	\end{equation}
	here and below comparison of vectors should be understood element-wise.
	Notice that due to continuity of multivariate normal distribution and assumption \eqref{garb}  $\forall \zb \in \Zcb$ 
	\begin{equation}\label{zbbound}
	\abs{\Pb(\zb) - (1-\alpha)}  \le \Rboot.
	\end{equation}
	Now for all $\zm \in \bZm$ and for all $\zb \in \Zcb$ it holds that 
	
	\begin{equation}
	\begin{split}
		\Pb(\zm) &\le P(\zm) +R \\ &\le \NN(\zm) +R+\Ra \\ &\le 1-\alpha - \Rboot \\ &\le \Pb(\zb)
	\end{split}
	\end{equation}
	where we have consequently used \eqref{tv}, \eqref{gar}, \eqref{bzmdef} and \eqref{zbbound}.
	In the same way one obtains for all $\zp \in \bZp$ and for all $\zb \in \Zcb$ 
	\begin{equation}
	\Pb(\zp) \ge   \Pb(\zb)
	\end{equation}
	which implies that $\Zcb \subset \Zm \cap \Zp $.

	Now denote quantile functions of $\infnorm{\eta_n}$ as $z^N : [0,1] \rightarrow \R^{\abs{\n}} $: 
	\begin{equation}
	\forall n \in \n : \Prob{\infnorm{\eta_n} \ge  z^N_n(\x) } = \x.
	\end{equation}
	In exactly the same way define quantile functions $\znb :[0,1] \rightarrow \R^{\abs{\n}} $ of $\infnorm{\zeta_n}$.
	Clearly for all $\x \in [0,1]$,  
	\begin{equation}
	z^N(\x + \delta) \le  \zb(\x) \le z^N(\x - \delta)
	\end{equation}
	and hence
	\begin{equation}
	\zb(\alphacorrected) \le  z^N(\alphacorrected - \delta) \le \zb(\alphacorrected - 2\delta),
	\end{equation}
	\begin{equation}
	1-\alpha \le \Pb{( z^N(\alphacorrected - \delta))} \le \Pb{(\zb(\alphacorrected - 2\delta))}.
	\end{equation}
	Using Taylor expansion with Lagrange remainder term we obtain for some $0 \le \kappa \le 2\delta$ 
	\begin{equation}
	\begin{split}
	\Nb\left(\zb(\alphacorrected - 2\delta)\right) &\le \Nb\left(\znb(\alphacorrected - 2\delta)\right) +\delta \\&= \Nb\left( \znb(\alphacorrected )\right) + \sum_{n \in \n} \partial_{z_n^\flat}\Nb(\znb(\alphacorrected )) \partial_\alpha \znb_n(\alpha^*)\kappa  +\delta\\
	&\le 1-\alpha + \sum_{n \in \n} \partial_{z_n^\flat}\Nb(\znb(\alphacorrected )) \partial_\alpha \znb_n(\alpha^*)\kappa + 3\delta.
	\end{split}
	\end{equation}
	Next successively using \ref{l21} and the fact that the quantile function is an inverse function of c.d.f. we obtain
	\begin{equation}
\Nb\left(\zb(\alphacorrected - 2\delta)\right)	\le 1-\alpha +3\delta+ 2\delta \abs{\n}
	\end{equation}
	and therefore 
	\begin{equation}
	1-\alpha \le \Pb\left(\zb(\alphacorrected - 2\delta)\right) \le 1-\alpha + \delta \left(3+2\abs{\n}\right),
	\end{equation}
	\begin{equation}
	1-\alpha \le \Pb\left(z^N(\alphacorrected - \delta)\right) \le 1-\alpha + \delta \left(3+2\abs{\n}\right).
	\end{equation}
	In the same way one obtains 
		\begin{equation}
		1-\alpha - \delta \left(3+2\abs{\n}\right) \le \Pb\left(z^N(\alphacorrected + \delta)\right) \le 1-\alpha.
		\end{equation}
	Next, by the argument used in the beginning of the proof we obtain 
	\begin{equation}
	z^N(\alphacorrected + \delta), z^N(\alphacorrected - \delta) \in \Zm( \delta \left(3+2\abs{\n}\right))\cap \Zp \left(\delta \left(3+2\abs{\n}\right)\right).
	\end{equation}
	As the final ingredient, we need to choose deterministic $\alphap$ and $\alpham$ such that 
	\begin{equation}
	N(z^N(\alpham + \delta))  = 1-\alpha-\delta \left(3+2\abs{\n}\right),
	\end{equation}
	\begin{equation}
	N(z^N(\alphap - \delta))  = 1-\alpha+\delta \left(3+2\abs{\n}\right)
	\end{equation}
(which is possible due to continuity),
	so $\alpham \le \alphacorrected \le \alphap$ and hence by monotonicity
	\begin{equation}
	z^N(\alpham + \delta) \le z^N(\alphacorrected + \delta)\le \zb(\alphacorrected) \le z^N(\alphacorrected - \delta) \le z^N(\alphap - \delta)
	\end{equation}
	and finally
	
	\begin{equation}
	\begin{split}
	1-\alpha-\delta \left(3 +2\abs{\n}\right)&\le  P(z^N(\alpham + \delta))\\  
	&\le P(\zb(\alphacorrected)) \\ 
	&\le P(z^N(\alphap - \delta)) \\ 
	&\le 1-\alpha+\delta \left(3 +2\abs{\n}\right).
	\end{split}	
	\end{equation}	
\end{proof}

\begin{lemma}\label{l21}
	Consider a random variable $\xi$ and an event $A$ defined on the same probability space. Let c.d.f. $\Prob{\xi \le x}$ and $\Prob{\xi \le x \& A}$ be differentiable. Then 
	\begin{equation}
	\frac{\partial_x \Prob{\xi \le x \& A} }{\partial_x\Prob{\xi \le x}} \le 1
	\end{equation}
\end{lemma}
\begin{proof}
	Indeed, denoting the complement of set $A$ as $\overline{A}$ we obtain,  
	\begin{equation}
	\begin{split}
		\frac{\partial_x \Prob{\xi \le x \& A} }{\partial_x\Prob{\xi \le x}} & = \frac{\partial_x \Prob{\xi \le x \& A} }{\partial_x\left(\Prob{\xi \le x \&  A} +\Prob{\xi \le x \& \overline{A}}\right) }\\
		 & = \frac{\partial_x \Prob{\xi \le x \& A} }{\partial_x \Prob{\xi \le x \&  A} +\partial_x  \Prob{\xi \le x \& \overline{A}} }\\
 		 & = \frac{1}{1 +\frac{\partial_x  \Prob{\xi \le x \& \overline{A}}}{\partial_x \Prob{\xi \le x \&  A}} }
	\end{split}
	\end{equation}
	Using the fact that derivative of c.d.f. is non-negative we finalize the proof.
\end{proof}

\section{Similarity of joint distributions of $\{A_n\}_{n \in \n}$ and $\{A_n^\flat\}_{n \in \n}$} \label{AAbSim}

\begin{lemma}\label{precTVlemma}
	Under assumptions of \ref{main} it holds that on set $\Tau$ with probability at least
	\begin{equation}
	1-\probsigmay-\probSigmaBound{s}{\gamma} - \probMoment{x}{t}
	\end{equation}
	that
	\begin{equation} 
	\sup_{\{x_n\}_{n \in \n} \subset \R} \abs{\Prob{\forall n \in \n : \An \le x_n} - \Probboot{\forall n \in \n : \Anb \le x_n}} \le \Ra + \Rboot + \Rsigma.
	\end{equation} 
\end{lemma}

\begin{proof}
	The proof consists in applying \ref{SbGarRandomBounds}, \ref{SGar}, \ref{twomaticies} and \ref{gaussianComparison}.
\end{proof}

\section{Gaussian approximation result for $A_n$} \label{GAprA}
\begin{lemma}\label{AGar}
	Suppose there exists $\Rt$ such that $\sqrt{n}\infnorm{r^\s(t)} \le \Rt$ for all $\s$ and $t$ on some set $\Tau$.	
	Then on set $\Tau$ it holds that 
	
	\begin{equation}
	\begin{split}
	\sup_x&\abs{\Prob{\forall n \in \n : \An \le \xn} -  \Prob{\forall n \in \n : \infnorm{\eta^n} \le \xn}} \le \Ra \\&\coloneqq \CA
	\left(\left(F \log^7(p^2T \nmax) \right)^{1/6} + 4\Rt\sqrt{\log (2p^2T)}\right) .
	\end{split}
	\end{equation}	
	where $F$ is defined by \eqref{defF} and $\eta^n$ by \eqref{etadef}.
	
\end{lemma}

\begin{proof}
	
	Substituting \eqref{Texpantion} to \eqref{Adef} yields
	
	\begin{equation}
	A_n(t) = \infnorm{ \underbrace{\frac{1}{\sqrt{2n}}  \inv{S}\left(  \sum_{i \in \Il} \Zv_i -  \sum_{i \in \Ir} \Zv_i \right)}_{\Szn(t)}  +  \frac{1}{\sqrt{2}} (\overline{r^l_n} - \overline{r^r_n}) }.
	\end{equation}
	Now denote stacked $\Szn(t)$ for all $n \in \n$ and as $\Szn$ and for all $n$  as $\Sz$.
	\ref{SGar} bounds the c.d.f. of $\infnorm{\Sz}$ as
	
	\begin{equation}
	\sup_x\abs{\Prob{\forall n \in \n : \infnorm{S^n_Z} \le \xn} - \Prob{\forall n \in \n : \infnorm{\eta^n} \le \xn}}  \le\CA
	\left(F \log^7(p^2T \nmax) \right)^{1/6}.
	\end{equation}
	But clearly on set $\Tau$
	
	\begin{equation}
	\abs{A_n - \infnorm{\Sz^n}} \le \sqrt{2} \Rt
	\end{equation}
	And hence for all $\{\xn\}_{n \in \n} \subset \R$
	
	\begin{equation}
	\begin{split}
	\abs{\Prob{\forall n \in \n :  \An < \xn| \Tau} - \Prob{\forall n \in \n : \infnorm{\eta^n} \le \xn}} &\le \CA
	\left(F \log^7(p^2T \nmax) \right)^{1/6} \\&+   \Prob{\forall n \in \n : \infnorm{\eta^n} \le \xn+ \sqrt{2}\Rt} \\&-  \Prob{\forall n \in \n : \infnorm{\eta^n} \le \xn- \sqrt{2}\Rt}.
	\end{split}
	\end{equation}
	Now notice that $\forall i : (\trueSigmaY)_{ii} = 1$ and bound the latter two terms by means of  \ref{anti}:
	
		\begin{equation}
		\begin{split}
		\sup_{\{\xn\}_{n \in \n} \subset \R^{\abs{\n}}}\abs{\Prob{\forall n \in \n :  \An < \xn| \Tau} - \Prob{\forall n \in \n : \infnorm{\eta^n} \le \xn}}&\le \CA
		\left(F\log^7(p^2T \nmax) \right)^{1/6} \\&+  
		4\Rt(\sqrt{\log (2p^2T)})
		\end{split}
		\end{equation}
\end{proof} 
\begin{lemma}\label{SGar}
	Let \ref{subGaussianVector} hold. Then 
	
	\begin{equation}
	\sup_x\abs{\Prob{\forall n \in \n : \infnorm{\Szn} \le \xn} -  \Prob{\forall n \in \n : \infnorm{\eta^n} \le \xn}}  \le \CA
	\left( F\log^7(2p^2T \nmax) \right)^{1/6}
	\end{equation}
	Where 
		\begin{equation}\label{etadef}
		\left(\begin{array}{cccc}
		\eta^1 &\eta^2 &... &\eta^{\abs{\n}} 
		\end{array}\right) \sim \N{0}{\trueSigmaY},
		\end{equation}
	
	\begin{equation}
	\trueSigmaY = \frac{1}{N} \sum_{i=1}^{N} \Var{Y_{\cdot i}},
	\end{equation}

	\begin{equation}\label{defF}
	F = \frac{1}{2\nmin}\left(\beta \log 2 \vee \frac{\sqrt{2}}{\sqrt{2} - 1} \gamma\right)^2 \vee \frac{1}{2\nmax}\left( {\frac{\nmax}{\nmin}}\right)^{1/3} M_3^2 \vee    \sqrt{\frac{1}{2\nmax\nmin}}   M_4^2
	\end{equation}
	with $\gamma$ defined by \eqref{gammadef}, $\beta$ by \eqref{betadef} and $Y$ by \eqref{BFM111} and an independent constant $\CA$ .
	
\end{lemma}

\begin{proof}
	Consider a matrix $Y_n$ with $2\nmax$ columns
	
		\begin{equation} \label{BFM}
		\begin{split}
		Y_n^T&\coloneqq \sqrt{\frac{\nmax}{n}} \times \\
&		\left( \begin{array}{cccccc}
	\Znormalized_1 & O & ... & O & -\Znormalized_{2\nmax+1} & ... \\
	\Znormalized_2 & \Znormalized_2 & ... & ... & ... & ... \\
	... & \Znormalized_3 & ... & ... & ... & ... \\
	\Znormalized_n & ... & ... & ... & ... & ... \\
	-\Znormalized_{n+1} & \Znormalized_{n+1} & ... & ... & ... & ... \\
	-\Znormalized_{n+2} & -\Znormalized_{n+2} & ... & ... & ... & ... \\
	... & -\Znormalized_{n+3} & ... & O & ... & ... \\
	-\Znormalized_{2n} & ... & ... & \Znormalized_{2\nmax-2n+1} & O & ... \\
	O & -\Znormalized_{2n+1} & ... & \Znormalized_{2\nmax-2n+2} & \Znormalized_{2\nmax-2n+2} & ... \\
	O & O & ... & ... & ... & ... \\
	... & ... & ... & -\Znormalized_{2\nmax - 1} & -\Znormalized_{2\nmax - 1} & ... \\
	O & O & ... & -\Znormalized_{2\nmax} & -\Znormalized_{2\nmax} & ...  \\\end{array} \right)
		\end{split}
		\end{equation}
	where $\Znormalized_i \coloneqq (\inv{S} \Zv_i)^T$.
	Clearly, columns of the matrix are independent and 
	
	\begin{equation}
	\Szn = \frac{1}{\sqrt{2\nmax}} \sum_{l=0}^{2\nmax} (Y_n)_{\cdot l}
	\end{equation}
	Next define a block matrix composed of $Y_n$ matrices:
	
	\begin{equation}\label{BFM111}
	Y \coloneqq \left(\begin{array}{c}
	Y_1 \\ \hline
	Y_2 \\ \hline
	... \\ \hline
	Y_{\numn} 
	\end{array}\right)
	\end{equation}
	Clearly vectors $Y_{\cdot l}$ are independent and 
	
	\begin{equation}
	\Sz = \frac{1}{\sqrt{2\nmax}}\sum_{l=0}^{2\nmax} Y_{\cdot l}
	\end{equation} 
	In order to complete the proof we make use of \ref{generalGAR}. 
	Denote 
	
	\begin{equation}\label{Bndef}
	B_{\nmax} = \sqrt{\frac{\nmax}{ \nmin}}\left(\beta \log 2 \vee \frac{\sqrt{2}}{\sqrt{2} - 1} \gamma\right) \vee \left( {\frac{\nmax}{ \nmin}}\right)^{1/6} M_3 \vee  \left({\frac{\nmax}{ \nmin}}\right)^{1/4} M_4
	\end{equation}
	
	By means of \ref{ZconcentrationComponentWise} one shows that the assumptions of \ref{expmoment} hold for components of $\Znormalized_i$ with 
	 
	\begin{equation} \label{gammadef}
	\gamma \coloneqq 12L^2 \sqrt{d}\maxLambda{\trueTheta} \infnorm{\trueTheta}   \infnorm{\inv{S}}
	\end{equation}
	
	\begin{equation} \label{betadef}
	\beta \coloneqq \left(\frac{9}{2}L^2 \sqrt{d}\maxLambda{\trueTheta}  + 1\right) \infnorm{\trueTheta}  \infnorm{\inv{S}}
	\end{equation}
	where $\maxLambda{\trueTheta}$ denotes the maximal eigen value of $\trueTheta$.
	Therefore condition \eqref{garassumption3} holds with $B_n$ defined by equation \eqref{Bndef}.

	\begin{equation}
	\frac{1}{N} \sum_{i=1}^{N} \E{(Y_{ij}^{n})^2} \ge  \min_j \Var{\Znormalized_{1j}}  = 1
	\end{equation}
	
	Hence, \ref{garass1} is fulfilled with $b = 1$. 
	Next notice that for some $k$-th component of $\Znormalized_i$ and central point $t$ (both defined by $j$):
	
	\begin{equation}
	\begin{split}
	\frac{1}{2\nmax} \sum_{i=1}^{2\nmax} \E{\abs{Y_{ij}^n}^3} & =  \frac{1}{2\nmax} \sum_{i \in \Il \cup\Ir }  \E{\left(\sqrt{\frac{\nmax}{n}} \abs{\Znormalized_{ik}}\right)^3 } \\
	& = \frac{1}{2\nmax} \sum_{i \in \Il \cup\Ir } \left(\frac{\nmax}{n}\right)^{3/2} \E{\abs{\Znormalized_{ik}}^3} \\
	& = \frac{2n}{2\nmax} \left(\frac{\nmax}{n}\right)^{3/2} \E{\abs{\Znormalized_{ik}}^3} \\
	& = \sqrt{\frac{\nmax}{n}} \E{\abs{\Znormalized_{ik}}^3}\\
	&\le \sqrt{\frac{\nmax}{\nmin}} M_3^3
	\end{split}
	\end{equation}
	and in the same way:
	
	\begin{equation}
	\frac{1}{2\nmax} \sum_{i=1}^{N} \E{\abs{Y_{ij}^n}^4} \le \frac{\nmax}{ \nmin} M_4^4
	\end{equation}
	Therefore \ref{garassumptions} holds with $B_{\nmax}$
	so \ref{generalGAR} applies here and provides us with the claimed bound. Moreover, 
	$\CA$ depends only on $b$ which equals one which implies that the constant $\CA$ depends on nothing.

\end{proof}
\begin{lemma} \label{expmoment}
	Consider a random variable $\xi$. Suppose the following bound holds $\forall x \ge 0$: 
	
	\begin{equation}
	\Prob{\abs{\xi}\ge \gamma x + \beta}  \le \ex{-x}
	\end{equation}
	Then 
	
	\begin{equation}
	\E{\exp\left( \frac{\abs{\xi}}{B} \right)} \le 2
	\end{equation}
	for
	 
	\begin{equation}
	B = \beta \log 2 \vee \frac{\sqrt{2}}{\sqrt{2} - 1} \gamma
	\end{equation}
	 
\end{lemma}

\begin{proof}
	
	Integration by parts yields
		
	\begin{equation}
	\E{\exp\left( \frac{\abs{\xi}}{B} \right)} \le \exp\left(\frac{\beta}{B}\right) + \frac{\gamma}{B} \int_{0}^{+\infty} \exp \left(\frac{\gamma x + \beta}{B}\right) \ex{-x} dx
	\end{equation}
	
	\begin{equation}
			\int_{0}^{+\infty} \exp \left(\frac{\gamma x + \beta}{B}\right) \ex{-x} dx 
			 = \frac{B}{B - \gamma } \exp\left(\frac{\beta}{B}\right) 
	\end{equation}

	\begin{equation}\begin{split}
	\E{\exp\left( \frac{\abs{\xi}}{B} \right)} & \le \frac{B}{B - \gamma } \exp\left(\frac{\beta}{B}\right) \\
	& \le 2
	\end{split}
	\end{equation}
	
\end{proof}

By the same technique the following lemma can be proven

\begin{lemma}\label{thirdAndFourthMomentBounds}
	Under assumptions of \ref{expmoment}
	
	\begin{equation}
		\E{\abs{\xi}^3} \le \beta^3 + 3\gamma\beta^2 + 6 \beta\gamma^2+2\gamma^3,
	\end{equation}
	
	\begin{equation}
		\E{\xi^4} \le \beta^4 + 4\gamma\beta^3 + 12 \beta^2 \gamma^2 6\beta\gamma^3 + 24 \gamma^4.
	\end{equation}
	
\end{lemma}

\section{Gaussian approximation result for $\Ab_n$} \label{GaprAB}
\begin{lemma}\label{SbGarRandomBounds}
		
		\begin{equation}
		\sup_{\{\xn\}_{n\in \n} \subset \R }\abs{\Probboot{\forall n \in \n : \Ab \le \xn} - \Probboot{\forall n \in \n : \infnorm{\zeta^n} \le \xn}}  \le \hatCAb
		\left(F^\flat \log^7(2p^2T \nmax) \right)^{1/6}.
		\end{equation}
		Where 
		
		\begin{equation}
		\left(\begin{array}{cccc}
		\zeta^1 &\zeta^2 &... &\zeta^{\abs{\n}} 
		\end{array}\right) \sim \N{0}{\hatSigmaY},
		\end{equation}
		
		\begin{equation}
		\hatSigmaY = \frac{1}{N} \sum_{i=1}^{N} \Var{\Yb_{\cdot i}},
		\end{equation}

		\begin{equation}
		F^\flat = \left(\frac{1}{2\nmin\log^22}  \vee \frac{1}{2\nmax}\left({\frac{\nmax}{\nmin}}\right)^{1/3} \vee  \sqrt{\frac{1}{2\nmax\nmin}}  \right)\infnorm{\inv{S}}^2 (M^\flat)^2
		\end{equation}
		
		\begin{equation}
		M^\flat = \max_{i \in \Istable} \infnorm{\hat Z_i} 
		\end{equation}	
		$\Ynb$ are defined by \eqref{BFMbootstrap}, and $\hatCAb$ depends only on $\min_{1 \le k \le p} (\hatSigmaY)_{kk}$
		
\end{lemma}

\begin{proof}
	Denote the term under the sign of $\infnorm{\cdot}$ in \eqref{bootAnDef} as $\Snbz$
	
	\begin{equation}
	\Snbz \coloneqq \frac{1}{\sqrt{2n}} \left( \sum_{i \in \Il} \Zbnormalized_i  - \sum_{i \in \Ir} \Zbnormalized_i \right)^T 
	\end{equation}
where $\Zbnormalized_i \coloneqq (\inv{S} \overline{Z^{\flat}_i})^T $ and let $\Sbz$ be a vector composed of stacked vectors $\Snbz$ for all $n \in \n$.

Consider a matrix 

\begin{equation}   \label{BFMbootstrap}
\begin{split}
(\Yb_n)^T&\coloneqq \sqrt{\frac{\nmax}{n}} \times\\ 
&\left( \begin{array}{cccccc}
\Zbnormalized_1 & O & ... & O & -\Zbnormalized_{2\nmax+1} & ... \\
\Zbnormalized_2 & \Zbnormalized_2 & ... & ... & ... & ... \\
... & \Zbnormalized_3 & ... & ... & ... & ... \\
\Zbnormalized_n & ... & ... & ... & ... & ... \\
-\Zbnormalized_{n+1} & \Zbnormalized_{n+1} & ... & ... & ... & ... \\
-\Zbnormalized_{n+2} & -\Zbnormalized_{n+2} & ... & ... & ... & ... \\
... & -\Zbnormalized_{n+3} & ... & O & ... & ... \\
-\Zbnormalized_{2n} & ... & ... & \Zbnormalized_{2\nmax-2n+1} & O & ... \\
O & -\Zbnormalized_{2n+1} & ... & \Zbnormalized_{2\nmax-2n+2} & \Zbnormalized_{2\nmax-2n+2} & ... \\
O & O & ... & ... & ... & ... \\
... & ... & ... & -\Zbnormalized_{2\nmax - 1} & -\Zbnormalized_{2\nmax - 1} & ... \\
O & O & ... & -\Zbnormalized_{2\nmax} & -\Zbnormalized_{2\nmax} & ... \\
\end{array} \right)
\end{split}
\end{equation}
which is a bootstrap counterpart of $Y_n$ from the proof of \ref{SGar} and construct a block  matrix $\Yb$ :

\begin{equation}
\Yb = \left(\begin{array}{c}
\Yb_1 \\ \hline
\Yb_2 \\ \hline
... \\ \hline
\Yb_{\numn} 
\end{array} \right)
\end{equation}
Clearly vectors $\Yb_{\cdot l}$ are independent and 

\begin{equation}
\Sbz = \frac{1}{\sqrt{2\nmax}}\sum_{l=0}^N \Yb_{\cdot l}
\end{equation} 
Now notice 

\begin{equation}
\frac{1}{2\nmax} \sum_{i=1}^{N} \E{\abs{Y_{ij}}^3} \le \sqrt{\frac{\nmax}{\nmin}} \max_{i \in \Istable} \infnorm{\hat Z_i}^3\infnorm{\inv{S}}^3
\end{equation}

\begin{equation}
\frac{1}{2\nmax} \sum_{i=1}^{N} \E{\abs{Y_{ij}}^4} \le \frac{\nmax}{\nmin} \max_{i \in \Istable}   \infnorm{\hat Z_i}^4\infnorm{\inv{S}}^4
\end{equation}
And finally apply \ref{generalGAR}.

\end{proof}

\begin{lemma} \label{bootstrapGarLemma}
		Let $\hatTheta$ denote an estimator of $\trueTheta$ s.t. for some positive $r$
		
		\begin{equation}
		\infnorm{\trueTheta - \hatTheta} < r
		\end{equation}
		and $\trueTheta_{ij} = 0 \Rightarrow \hatTheta_{ij}=0$,
		furthermore, let $\deltaY(r) < 1/2$,
		also suppose \ref{subGaussianVector} holds. 	
	Then 	at least with probability $1-\probMoment{x}{t} - \probsigmay$
	
	\begin{equation}
			\sup_{\{\xn\}_{n\in \n} \subset \R }\abs{\Probboot{\forall n \in \n : \Ab \le \xn} - \Probboot{\forall n \in \n : \infnorm{\zeta^n} \le \xn}}   \le R_{A^b}  \coloneqq
		\CAb\left(\hat F \log^7(2p^2T \nmax) \right)^{1/6} 
	\end{equation}
	where 
	
	\begin{equation}
	\hat F = \left(\frac{1}{2\nmin\log^22}  \vee \frac{1}{2\nmax}\left({\frac{\nmax}{\nmin}}\right)^{1/3} \vee  \sqrt{\frac{1}{2\nmax\nmin}}  \right) \infnorm{\inv{S}}^2 (C^\flat)^2
	\end{equation}
	
	\begin{equation}
	C^{\flat} \coloneqq \Zbound{s}{x} + (3(dx)^2 + 1) r
	\end{equation}
	and constant $\CAb$ depends only on	 $\deltaY$.
	
\end{lemma}

\begin{proof}
	The proof consists in subsequently applying \ref{SbGarRandomBounds} and \ref{MomendBoundLemma} ensuring $C^\flat \ge M^\flat = \max_{i \in \Istable} \infnorm{\hat Z_i} $ with probability at least $ 1- \probMoment{x}{t}$ and applying \ref{twomaticies} providing that $\infnorm{\trueSigmaY - \hatSigmaY} \le \deltaY \le 1 = \min_{1 \le k \le p} (\trueSigmaY)_{kk}$ with probability at least $1-\probsigmay$ which implies the existence of a deterministic constant $\CAb > \hatCAb$.
\end{proof}

\begin{lemma} \label{MomendBoundLemma}
	Let $\hatTheta$ denote an estimator of $\trueTheta$ s.t. for some positive $r$
	
	\begin{equation}
	\infnorm{\trueTheta - \hatTheta} < r
	\end{equation}
	and $\trueTheta_{ij} = 0 \Rightarrow \hatTheta_{ij}=0$.
	Also let \ref{subGaussianVector} hold.
	Then with probability at least $1-\probMoment{x}{t} $
	
	\begin{equation} \label{mexp}
	M^\flat  \le \Zbound{s}{x} + \deltaZ
	\end{equation}
	where $\probMoment{x}{x} \coloneqq \probzbound{s}{x} + \probXBound$.
	
\end{lemma}

\begin{proof}
	Direct application of \ref{Zconcentration} yields 
	
	\begin{equation}
		\Prob{\forall i \in \Istable  : \infnorm{Z_i} \le \Zbound{s}{x}} \ge  1-\probzbound{s}{x}
	\end{equation}
	which in combination with the fact (provided by \ref{zmistie}) that $ \infnorm{\hat Z_i - Z_i} \le \deltaZ$ implies \eqref{mexp}.
	
\end{proof}

\section{$\hatSigmaY \approx \trueSigmaY$ } \label{twomatsec}
First of all, if $\trueSigmaZ \coloneqq \Var{\Zvi} \approx \Varboot{\overline{\Zbi}}$, then $\trueSigmaY \approx \hatSigmaY$ as well (\ref{twomaticies}).
The idea is to notice that 
\begin{equation}
\Var{\overline{\Zbi}} = \hatSigmahatZ \coloneqq \empE{\left(\Zhatvi - \empE{\Zhatvi}\right) \left(\Zhatvi - \empE{\Zhatvi}\right)^T} 
\end{equation}
due to the choice of the bootstrap scheme.
Next we show that $\trueSigmaZ  \approx \hatSigmaZ \coloneqq \empE{\left(\Zvi - \empE{\Zvi}\right) \left(\Zvi - \empE{\Zvi}\right)^T}$ (\ref{sigmaDelta0}) and finalize the proof by proving that $\hatSigmaZ \approx \hatSigmahatZ$ (\ref{ZCovConcentration}).

The results of this section rely on a lemma which is a trivial corollary of Lemma 6 by \cite{sara} providing the concentration result for the empirical covariance matrix 

\begin{lemma}
	\label{sigmaconcentration}
	Let \ref{subGaussianVector} hold for some $L > 0$. Then for any positive $\gamma$
	\begin{equation}
	\delta_n(\chi) \coloneqq 2L^2 \left(\frac{2 \log p + \chi}{n} + \sqrt{\frac{4\log p + 2\chi}{n}}\right)
	\end{equation}
	
	\begin{equation}
	\Prob{\infnorm{\hat{\Sigma} - \Sigma^*} \ge \delta_n(\gamma) } \le \probSigmaBound{n}{\gamma} \coloneqq 2\ex{-\chi}.
	\end{equation}
\end{lemma}

\begin{lemma}\label{twomaticies}
	Assume, \ref{subGaussianVector} holds. Moreover, let
	
	\begin{equation}
	\infnorm{\empE{X_i X_i^T} - \trueSigma} \le \delta_s
	\end{equation}
	and let $\hatTheta$ denote a symmetric estimator of $\trueTheta$ s.t. 
	
	\begin{equation}
	\infnorm{\trueTheta - \hatTheta} < r
	\end{equation}
	and $\trueTheta_{ij} = 0 \Rightarrow \hatTheta_{ij}=0$.
	Then for positive $x$ and $q$
	
	\begin{equation}
	\Prob{\infnorm{\hatSigmaY - \trueSigmaY} \ge \deltaY} \le \probsigmay
	\end{equation}
	where 
	
	\begin{equation}
	\probsigmay \coloneqq \probsigmazboundone + \probsigmazboundtwo
	\end{equation}
	
	\begin{equation}
		\deltaY \coloneqq \infnorm{S^{-1}}^2\left(\SigmaZDelta^{(1)} + \SigmaZDelta^{(2)}\right)
	\end{equation}
	and $\SigmaZDelta^{(1)}$ and $\SigmaZDelta^{(2)}$ along with the probabilities $\probsigmazboundone$ and $\probsigmazboundtwo$  are defined in \ref{sigmaDelta0} and \ref{ZCovConcentration} respectively.

\end{lemma}

\begin{proof}
	 Notice that 
	 
	 \begin{equation}
	 \infnorm{\hatSigmaY - \trueSigmaY} = \infnorm{\inv{S}\hatSigmahatZ\inv{S} - \inv{S}\trueSigmaZ\inv{S}}\le \infnorm{S^{-1}}^2\infnorm{\hatSigmahatZ - \trueSigmaZ}
	 \end{equation}
	 because the matrices $\hatSigmaY$ and  $\trueSigmaY$ are composed of blocks $\inv{S}\hatSigmaZ\inv{S}$ and $\inv{S}\trueSigmaZ\inv{S}$ respectively, each block multiplied by some positive value not greater than $1$ (which can be verified by simple algebra).
	 
	 	By \ref{ZCovConcentration} and \ref{sigmaDelta0}
	 	
	 	\begin{equation}
	 	\infnorm{\hatSigmahatZ - \trueSigmaZ} \le \SigmaZDelta^{(1)} + \SigmaZDelta^{(2)}
	 	\end{equation}
	 	and hence 
	 	
	 	\begin{equation}
	 	\infnorm{\hatSigmaY - \trueSigmaY} \le \infnorm{S^{-1}}^2(\SigmaZDelta^{(1)} + \SigmaZDelta^{(2)})
	 	\end{equation} 
	 	with probability at least 
	 	\begin{equation}
	 	1-\probsigmazboundone-\probsigmazboundtwo
	 	\end{equation}
	 	
\end{proof}
\begin{lemma} 
	\label{ZconcentrationComponentWise}
	Under \ref{subGaussianVector} it holds for arbitrary $1 \le u,v\le p$ and positive $x$  that
	\begin{equation}
	\Prob{\abs{Z_{1,uv}} \le \left(3L^2 \sqrt{d}\maxLambda{\trueTheta} \left(\frac{3}{2}+ 4x  \right) + 1\right) \infnorm{\trueTheta}    } \ge  1-\ex{-x}
	\end{equation}
	
\end{lemma}

\begin{proof}
	Re-write the definition \eqref{zijk} of an element $\Ziuv$ for arbitrary $1 \le u,v \le p$

	\begin{equation}
	\begin{split}
	\Ziuv &= \trueTheta_u X_i\trueTheta_vX_i-\trueTheta_{uv}\\
	& = X_i^T \left[\trueTheta_u  (\trueTheta_v)^T\right] X_i - \trueTheta_{uv}.
	\end{split}
	\end{equation}
	
	The first term is clearly a value of a quadratic form defined by the matrix $B=\trueTheta_u  (\trueTheta_v)^T$. Note that $\rank B=1$ which implies that it is either positive semi-definite or negative semi-definite. Next we apply \ref{qf} and obtain for all positive $x$
	
	\begin{equation}\label{bound}
		\Prob{\abs{X_i^T B X_i} \ge 3L^2\left( \abs{\tr B} + 2 \sqrt{\tr(B^2)x} + 2 \abs{\maxLambda{B}}x  \right) } \le \ex{-x}.
	\end{equation}
	Again, due to the fact that $B$ is a rank-$1$ matrix 
	
	\begin{equation}\label{eqs}
	\tr B = \maxLambda{B} = \sqrt{\tr B^2}
	\end{equation}
	and by construction of matrix $B$
	
	\begin{equation}\label{treq}
	\begin{split}
	\abs{\tr B}  &= \abs{\trueTheta_u (\trueTheta_v)^T} \\ 
	& \le \onenorm{\trueTheta_u} \infnorm{\trueTheta} \\ 
	& \le \sqrt{d} \twoNorm{\trueTheta_u} \infnorm{\trueTheta} \\ 
	& \le \sqrt{d}\maxLambda{\trueTheta}\infnorm{\trueTheta}.
	\end{split}
	\end{equation}
	Substitution of \eqref{eqs} and \eqref{treq} to \eqref{bound} yields 
	
	\begin{equation}
	\Prob{\abs{X_i^T B X_i} \ge 3L^2 \sqrt{d}\maxLambda{\trueTheta} \infnorm{\trueTheta} \left(1 + 2  \sqrt{x} + 2  x  \right) } \le \ex{-x}.
	\end{equation}
	And since $\sqrt{x} \le x + \frac{1}{4}$
	
	\begin{equation}
	\Prob{\abs{X_i^T B X_i} \ge 3L^2 \sqrt{d}\maxLambda{\trueTheta} \infnorm{\trueTheta} \left(\frac{3}{2}  + 4x  \right) } \le \ex{-x}.
	\end{equation}
	Finally, we obtain a bound for $\Ziuv$ as 
	
	\begin{equation}
	\Prob{\abs{\Ziuv} \ge \left(3L^2 \sqrt{d}\maxLambda{\trueTheta} \left(\frac{3}{2}  + 4x  \right) + 1\right) \infnorm{\trueTheta}  } \le \ex{-x}.
	\end{equation}
	
\end{proof}

Correction for all $i, u \text{ and } v$ establishes the following result

\begin{lemma} 
	\label{Zconcentration}
	Consider an i.i.d. sample $X_i$ of length $n$. Under \ref{subGaussianVector} for positive $x$ it holds that
	\begin{equation}
	\Prob{\forall i \in \{1..n\}  : \infnorm{Z_i} \le \Zbound{n}{x}} \ge  1-\probzbound{n}{x}
	\end{equation} 
	where 
	
	\begin{equation}
	\Zbound{n}{x} \coloneqq	 \left(3L^2 \sqrt{d}\maxLambda{\trueTheta} \left(\frac{3}{2} +4 \log p^2n + 4x  \right) + 1\right) \infnorm{\trueTheta} ,
	\end{equation}
	
	\begin{equation}
	\probzbound{n}{x} \coloneqq	\ex{-x}.
	\end{equation}
\end{lemma}
\begin{lemma}\label{sigmaDelta0}
Under \ref{subGaussianVector} for positive $x$ and $q$

\begin{equation}
\Prob{\infnorm{\hatSigmaZ - \trueSigmaZ} \ge \SigmaZDelta^{(1)}} \le \probsigmazboundone
\end{equation} 
where 

\begin{equation}
\SigmaZDelta^{(1)} \coloneqq \frac{s}{s-1} \Bernstain{4\Zboundsqr{s}{x} + \frac{s-1}{s}\infnorm{\trueSigmaZ}}{q}{s}{\sigma_W^2} 
\end{equation}

\begin{equation}
\probsigmazboundone \coloneqq p^4\ex{-q} + \probzbound{s}{x}
\end{equation}

\end{lemma}

\begin{proof}
Denote 

\begin{equation}
\Wi \coloneqq (\Zv_i - \empE{\overline{Z_i}}) (\Zv_i - \empE{\overline{Z_i}})^T - \frac{s-1}{s} \trueSigmaZ
\end{equation} 
and note that 

\begin{equation}
\frac{s-1}{s}\left(\hatSigmaZ - \trueSigmaZ\right) = \frac{1}{s}\sum_{i\in \Istable} \Wi.
\end{equation}

By \ref{Zconcentration} we have $\infnorm{Z_i} \le \Zbound{s}{x}$ with probability at least $1-\probzbound{s}{x}$ which implies $\infnorm{\Wi} \le 4\Zboundsqr{s}{x} + \frac{s-1}{s} \infnorm{\trueSigmaZ}$. 
Since $\Wi_{kl}$ are i.i.d., bounded and centered, Bernstein inequality applies here:

\begin{equation}
\Prob{\empE{\Wi_{kl}} \ge \Bernstain{4\Zboundsqr{s}{x} + \frac{s-1}{s}\infnorm{\trueSigmaZ}}{q}{s}{\sigma_W^2} } \le e^{-q}
\end{equation} 
where $\sigma_{W}^2$ is the smallest variance of components of $\Wi$. Therefore

\begin{equation}
\Prob{\infnorm{\empE\Wi} \ge \Bernstain{4\Zboundsqr{s}{x} + \frac{s-1}{s}\infnorm{\trueSigmaZ}}{q}{s}{\sigma_W^2} } \le p^4e^{-q}.
\end{equation}

\end{proof}

The following lemma bounds the mis-tie between $Z_i$ and $\hat Z_i$.

\begin{lemma} \label{zmistie}
	Let \ref{subGaussianVector} holds
	and let $\hatTheta$ be a symmetric estimator of $\trueTheta$ s.t. 
	
	\begin{equation}
	\infnorm{\trueTheta - \hatTheta} < r
	\end{equation} 
	and $\trueTheta_{ij} = 0 \Rightarrow \hatTheta_{ij}=0$. Then for positive $x$
	
	\begin{equation}
		\Prob{\forall i\in \Istable: \infnorm{Z_i - \hat Z_i} \le \deltaZ} \ge 1 -  \probXBound
	\end{equation} 
	where 
	
	\begin{equation}
		\deltaZ \coloneqq 2rd^{3/2}x^2 \infnorm{\trueTheta} + (rdx)^2
	\end{equation}
	
	\begin{equation}
		\probXBound \coloneqq s\ex{-x^2/L^2}
	\end{equation}
	
\end{lemma}

\begin{proof}
	Due to sub-Gaussianity,
	
	\begin{equation} \label{xboundeq}
		\forall \alpha \in \R^p : \Prob{\abs{\alpha^T X_i} \le x} \ge 1- s\ex{-x^2/L^2}
	\end{equation} 
	Now consider the mis-tie of arbitrary elements $\Ziuv$ and $\hatZiuv$ : 
	
	\begin{equation}
	\begin{split}
	\abs{\Ziuv - \hatZiuv} &= \abs{\trueTheta_u X_i\trueTheta_v X_i + \trueTheta_{uv} - \hatTheta_u X_i\hatTheta_v X_i  -\hatTheta_{uv} } \\ 
	& \le \abs{(\trueTheta_{u} - \hatTheta_{u}) X_i\trueTheta_v X_i } + \abs{(\trueTheta_{u} - \hatTheta_{u}) X_i\hatTheta_v X_i } +r 
	\end{split}
	\end{equation} 
	Now note that due to \eqref{xboundeq} and assumptions imposed on $\trueTheta$
	
	\begin{equation}
	\abs{\trueTheta_v X_i} \le \sqrt{d} \infnorm{\trueTheta} x
	\end{equation}
	
	\begin{equation}
		\abs{(\trueTheta_v - \hatTheta_v) X_i} \le rdx
	\end{equation}
	
	\begin{equation}
		\abs{\hatTheta_v X_i} \le \abs{\trueTheta_v X_i} + \abs{(\trueTheta_v - \hatTheta_v) X_i} \le \sqrt{d} \infnorm{\trueTheta} x + rdx
	\end{equation} 
	And hence 
	
	\begin{equation}
	\abs{\Ziuv - \hatZiuv} \le 2rd^{3/2}x^2 \infnorm{\trueTheta} + (rdx)^2
	\end{equation}
	
\end{proof}

\begin{lemma}\label{ZCovConcentration}
	
	Assume \ref{subGaussianVector} holds.	
	Let $\hatTheta$ be a symmetric estimator of $\trueTheta$ s.t. 
	
	\begin{equation}
	\infnorm{\trueTheta - \hatTheta} < r
	\end{equation} 
	and $\trueTheta_{ij} = 0 \Rightarrow \hatTheta_{ij}=0$.
	Then for positive $x$
	
	\begin{equation}
	\Prob{\infnorm{\hatSigmaZ - \hatSigmahatZ} \ge \SigmaZDelta^{(2)}} \le \probsigmazboundtwo 
	\end{equation} 
	where
	
	\begin{equation}
		 \probsigmazboundtwo \coloneqq \probXBound + \probzbound{s}{x}
	\end{equation}
	
	\begin{equation}
		\SigmaZDelta^{(2)} = \deltaZ (2\Zbound{s}{\x} + \deltaZ)
	\end{equation}
\end{lemma}

\begin{proof}

	By \ref{Zconcentration} with probability at least $1-\probzbound{s}{\x}$ we have $\infnorm{Z_i} \le \Zbound{s}{\x}$ and in combination with \ref{zmistie} we obtain $\infnorm{\hatZvi} \le \Zbound{s}{\x} + \deltaZ$ with probability at least $1-\probzbound{s}{x} - \probXBound$.
	Now denote 
	
	\begin{equation}
	\xii \coloneqq \Zvi - \empE{\Zvi } \text{ and } \hatxii \coloneqq \hatZvi - \empE{\hatZvi} 
	\end{equation} 
	And deliver the bound
	
	\begin{equation}
	\begin{split}
	\infnorm{\hatSigmaZ - \hatSigmahatZ} & \le \empE{\xii (\xii - \hatxii)^T + (\xii - \hatxii)\hatxii^T} \\
	& \le \left(\infnorm{\hatxii} + \infnorm{\xii}\right)  \infnorm{\xii - \hatxii} \ \\
	& \le \deltaZ (2\Zbound{s}{\x} + \deltaZ)
	\end{split}
	\end{equation}
	
\end{proof}

\section{Known results}\label{knownresults}

\subsection{Gaussian approximation result}
In this section we briefly describe the result obtained in \cite{Chernozhukov2014}.

Throughout this section consider an  independent sample $x_1, ... , x_n \in \R^p$ of centered random variables. Define their Gaussian counterparts  $y_i \sim \N{0}{\Var{x_i}}$ and denote their scaled sums as

\begin{equation}
S^X_n \coloneqq \frac{1}{\sqrt{n}} \sum_{i=1}^n x_i
\end{equation}

\begin{equation}
S^Y_n \coloneqq \frac{1}{\sqrt{n}} \sum_{i=1}^n y_i 
\end{equation}

\begin{definition}\label{hyperrectdef}
	We call a set $A$ of the form $A = \{w \in \R^p : a_i \le w_i \le b_i ~\forall i \in \{1..p\} \}$ a hyperrectangle. The family of all hyperrectangles is denoted as $A^{re}$.
\end{definition}

\begin{assumption} \label{garass1}
	$\exists b > 0$ such that
	
	\begin{equation} \label{garassumption1}
	\frac{1}{n} \sum_{i=1}^{n}\E{x_{ij}^2} \ge b \text{ for all } j \in 1..p
	\end{equation}
	
\end{assumption}

\begin{assumption} \label{garassumptions}
	$\exists G_n \ge 1$ such that 
	
	\begin{equation} \label{garassumption2}
	\frac{1}{n} \sum_{i=1}^{n} \E{\abs{x_{ij}}^{2+k}} \le G_n^{2+k} \text{ for all } j \in 1..p \text{ and } k \in \{1,2\}
	\end{equation}	
	
	\begin{equation} \label{garassumption3}
	\E{\exp \left(\frac{\abs{x_{ij}}}{G_n}\right)} \le 2 \text{ for all } j \in 1..p \text{ and } i \in 1.. n
	\end{equation}
	
\end{assumption}

\begin{lemma}[Proposition 2.1 by \cite{Chernozhukov2014}]\label{generalGAR}
	Let \ref{garass1} hold for some $b$ and \ref{garassumptions} hold for some $G_n$. 
	Then
	
	\begin{equation}
	\sup_{A \in A^{re}} \abs{\Prob{S^X_n \in A}  - \Prob{S^Y_n \in A}} \le C\left(\frac{G_n^2\log^7(pn)}{n}\right)^{1/6}
	\end{equation} 
	and the constant $C$ depends only on $b$.
	
\end{lemma}

%\subsection{Sparse precision matrix estimation} \input{glasso.tex}

\subsection{Anti-concentration result}
\begin{lemma}[Nazarov's inequality \cite{Nazarov2003}] \label{anti}
	Consider a normal $p$-dimensional vector $X \sim \N{0}{\Sigma}$ and let $\forall i : \Sigma_{ii} = 1$. Then for any $y \in \R^p$ and any positive $a$
	\begin{equation}
		\Prob{X \le y +a} - \Prob{X \le y} \le C a \sqrt{\log p},
	\end{equation}
	where $C$ is an independent constant.
\end{lemma} 
 
\subsection{Gaussian comparison result}

By the technique given in the proof of Theorem 4.1 by \cite{Chernozhukov2014} one obtains the following generalization of the result given in \cite{anticonc}

\begin{lemma} \label{gaussianComparison}
	Consider a pair of covariance matrices $\Sigma_1$ and $\Sigma_2$ of size $p\times p$ such that
	\begin{equation}
	\infnorm{\Sigma_1 - \Sigma_2} \le \Delta
	\end{equation} 
	and $\forall k : C_1 \ge \Sigma_{1,kk} \ge c_1 > 0$.
	Then for random vectors $\eta \sim \N{0}{\Sigma_1}$ and $\zeta \sim \N{0}{\Sigma_2}$ it holds that 
	
	\begin{equation}
	\sup_{A \in A^{re}} \abs{\Prob{\eta \in A} - \Prob{\zeta \in A}} \le C\Delta^{1/3} \log^{2/3}p,
	\end{equation} 
	where $C$ is a positive constant which depends only on $C_1$ and $c_1$.
\end{lemma}

\subsection{Tail inequality for quadratic forms}

The following result is a direct corollary of Theorem 1 in \cite{qboundPaper}

\begin{lemma}\label{qf}
	Consider a positive semi-definite or negative semi-definite matrix $B$ and suppose \ref{subGaussianVector} holds. Then for all $t>0$
	
	\begin{equation}
	\Prob{\abs{X_1^T B X_1} \ge 3L^2\left( \abs{\tr B} + 2 \sqrt{\tr(B^2) t} + 2 \abs{\maxLambda{B}}t  \right) } \le \ex{-t}
	\end{equation}
	
\end{lemma}

\subsection{High-dimensional precision matrix estimation}\label{glassosec}
In order to address the problem of high-dimensional precision matrix estimation one has to assume its sparsity. Below we describe two approaches exploiting this assumption. In both of them we assume that an i.i.d. sample $X_1, ... X_n \in \R^p$ is supplied.

\subsubsection{Graphical lasso}

In \cite{glassoTib} the graphical lasso approach was suggested. An estimate may be obtained as the solution of the following optimization problem over a positive-definite cone $S_{++}^p$ of $p \times  p$ dimensional matrices.

\begin{equation}\label{logdet}
\glestimation \coloneqq \arg\mymin{\Theta \in S_{++}^p}{ \tr(\Theta \hatSigma) - \log det \Theta + \lambda \onenorm{\Theta}}
\end{equation}
where $\hatSigma$ stands for the empirical covariance matrix 

\begin{equation}
\hatSigma = \frac{1}{n}\sum\limits_{i=1}^n X_i X_i^T.
\end{equation}

The theoretical treatment of the approach keeps track  on the following Schatten norms: $\kappaSigma = \InfNorm{  \Sigma^* }$ and $\kappaGamma = \InfNorm{ \inv{ (\Gamma^*_{SS}) }}$.
The following result establishes consistency of the estimator in the sense of \ref{consdef}.
\begin{lemma}[Theorem 1, \cite{ravikumar2011}] 
	\label{ravikumarTheorem}
	Consider a distribution satisfying \ref{irrepass} with some $\phi \in (0,1]$,  let $\hat\Theta$ be a solution of the optimization problem \eqref{logdet} with tuning parameter $\lambda_n = \frac{8}{\psi} \delta_n$.
	Furthermore, impose the following sparsity assumption:
	\begin{equation}\label{sparsity}
	d \le \frac{1}{6(\delta_n + \lambda_n)\max\{\kappaGamma\kappaSigma, \kappaGamma^2\kappaSigma^3  \}}.
	\end{equation}
	Then on the set  $\mathcal{T} = \left\{ \infnorm{\hat\Sigma - \Sigma^*} < \delta_n \right\}$the following holds:
	\begin{equation}
	\infnorm{\hat\Theta^{GL} - \Theta^*}  \le  r_{\lambda} \coloneqq 2\kappaGamma(\delta_n + \lambda_n)
	\end{equation}
	and
	\begin{equation}
	\Theta^*_{ij} = 0  \Rightarrow \hat\Theta_{ij} = 0 .
	\end{equation}
\end{lemma}

A rather similar result is provided in paper \cite{mypaper} for adaptive versions of graphical lasso suggested and studied in \cite{zou2008}   \cite{SCADGlasso} \cite{Fan01variableselection}  \cite{zou2006} .
\subsubsection{Node-wise lasso}

This section describes the node-wise lasso approach which was suggested in \cite{meinshausen2006}.

For each $1 \le j \le n$ define a vector $$\hat\Gamma_j \coloneqq \left(\hat\gamma_{j\,1}, ..., \hat\gamma_{j\,j-1}, 1, \hat\gamma_{j\,j+1} , ..., \hat\gamma_{j\,p}\right) $$ 
where $\hat{\gamma}_j$ is defined as a solution of the following lasso regression:
\begin{equation}
\hat\gamma_j  \coloneqq \arg\max_{\gamma \in \R^{p-1}} \frac{1}{n}\sum_{1 \le i \le n} \left(X_{ij} - X_{i,-j}^T\gamma \right)^2 + 2\lambda \onenorm{\gamma}
\end{equation}
and 
\begin{equation}
\hat{\tau}^2_j \coloneqq \frac{1}{n}\sum_{1 \le i \le n} \left(X_{ij} - X_{i,-j}^T\hat\gamma_j \right)^2 + \lambda \onenorm{\gamma}.
\end{equation}
Finally the $j$-th column of the estimator is defined as 

\begin{equation}
\mbestimation_{j} \coloneqq \hat{\Gamma}_j  / \hat\tau^2_j .
\end{equation}

Note, that this estimator might not be symmetric, so one cannot use it as an estimator $\hatTheta$ based on the sub-sample $\{X_i\}_{i\in\Istable}$. 
The paper \cite{janamb} suggests to construct a de-sparsified estimator $\hat T(\mbestimation)$ where

\begin{equation} \label{despdef}
\hat T(\hatTheta) \coloneqq  \hatTheta + \hatTheta^T- \hatTheta^T\hatSigma \hatTheta
\end{equation}
and threshold elements of $\hat T$ obtaining a positive-definite estimate. 

Under \ref{subGaussianVector}, the sparsity assumption $\frac{d\log p}{{n}} = o(1)$ and the assumption of the bounded spectrum (\ref{boundedEigen}) the paper \cite{janamb} establishes the root-$n$ consistency of such an estimator (see \ref{consdef}).

\begin{assumption}\label{boundedEigen}
	\begin{equation}\label{key}
	\exists E : \frac{1}{E} \le \minLambda{\trueTheta} \le \maxLambda{\trueTheta} \le E .
	\end{equation}
\end{assumption}

\subsubsection{Bounds for $r$}

While graphical lasso and node-wise estimate are point estimates, de-sparsified estimators have been suggested in order to obtain confidence intervals \cite{sara} \cite{janamb}. 

The analysis of these estimators relies on the bounds for the residual term $r$:

\begin{equation}
r \coloneqq {\hat T - \left(\trueTheta - \trueTheta (\trueSigma - \hatSigma) \trueTheta\right)}
\end{equation}

The next two lemmas bound the remainder $r$ for the case of graphical lasso and node-wise estimator.

\begin{lemma}[by \cite{sara}]
	\label{lemmar}
	Impose \ref{subGaussianVector}, \ref{irrepass}  and \ref{boundedEigen}.
	Then under the sparsity assumption 
	\begin{equation}\label{sparsityAss}
	\frac{d\log p}{\sqrt{n}} = o(1)
	\end{equation}
	it holds that
	
	\begin{equation}
	\infnorm{r} = O_p\left(\frac{d\log p}{{n}}\right).
	\end{equation}
	
\end{lemma}

A finite sample-size bound for $r$ along with its adaptations for the case of adaptive graphical lasso may be found in \cite{mypaper}

\begin{lemma}[by \cite{janamb}] \label{mbrem}
	Let $\hatTheta$ be yielded by the node-wise procedure with $\lambda_n \asymp \sqrt{\frac{\log p}{n}}$. Then under \ref{subGaussianVector}, \ref{boundedEigen} and sparsity assumption  \eqref{sparsityAss}
	
	\begin{equation}
	\infnorm{r} = O_p\left(\frac{d\log p}{{n}}\right).
	\end{equation}

\end{lemma}

\section*{Acknowledgements}
We thank Vladimir Spokoiny and Karsten Tabelow for comments and discussions that greatly improved the manuscript. We also thank Jun Li and Haeran Cho for providing the implementation of the approach proposed in \cite{twoCov} and in \cite{Cho2015} respectively. The last but not the least, we thank André Brechmann for providing the data of the study \cite{Puschmann} funded by the German Science Foundation (DFG).

The research of Valeriy Avanesov has been partially funded by Deutsche Forschungsgemeinschaft (DFG) - SFB1294/1 - 318763901.

\bibliographystyle{plain}
\bibliography{changePointPaper}

\end{document}